\documentclass[onenumber, plainnumbering,english,11pt]{amsart}
\usepackage[utf8]{inputenc}
\usepackage{scalerel,stackengine}
\usepackage[square,numbers]{natbib}
\usepackage{amssymb,amsmath,amscd,mathrsfs,paralist}
\usepackage[mathscr]{eucal}
\usepackage{bm}
\usepackage[T1]{fontenc}
\usepackage{mathtools, nccmath}
\usepackage{xpatch}
\xpatchcmd{\NCC@ignorepar}{%
\abovedisplayskip\abovedisplayshortskip}
{%
\abovedisplayskip\abovedisplayshortskip%
\belowdisplayskip\belowdisplayshortskip}
{}{}
\usepackage{graphicx}
\usepackage[all,cmtip]{xy}
\usepackage[colorlinks,linkcolor=black,urlcolor=black]{hyperref}
\hypersetup{
     colorlinks   = true,
     citecolor    = black}
\usepackage{tikz}
\usepackage{pgfplots}
\newcommand{\R}{\mathbb R}

\newcommand{\Z}{\mathbb Z}
\newcommand{\C}{\mathbb C}

\newcommand{\nc}{\newcommand}
\nc{\BCc}{{\mathbb{C}(\wp(z),\wp^\prime(z))}}
\nc{\BC}{{\mathbb C}}
\nc{\BQ}{{\mathbb Q}}
\nc{\BR}{{\mathbb R}}
\nc{\BZ}{{\mathbb Z}}
\nc{\BP}{{\mathbb P}}
\nc{\BN}{{\mathbb N}}
\nc{\BM}{{\mathbb M}}
\nc{\BF}{M}
\nc{\fH}{{\mathfrak{H}}}
\nc{\vp}{{\varepsilon}}\nc{\dpar}{{\partial}}\nc{\al}{{\alpha}}
\nc{\PSL}{PSL(2,\BR)}
\nc{\PS}{PSL(2,\BZ)}
 \nc{\CL}{PSL(2,\BZ/m\BZ)}
 \newtheorem{theorem}{Theorem}[section]
\newtheorem{corollary}[theorem]{Corollary}
\newtheorem{lemma}[theorem]{Lemma}
\newtheorem{proposition}[theorem]{Proposition}
\theoremstyle{remark}
\newtheorem{definition}[theorem]{Definition}
\newtheorem{remark}[theorem]{Remark}
\newtheorem{example}[theorem]{Example}

   \newcommand{\theoref}[1]{Theorem~\ref{#1}}
\newcommand{\propref}[1]{Proposition~\ref{#1}}

\newcommand{\lemref}[1]{Lemma~\ref{#1}}

\newcommand{\exref}[1]{Example~\ref{#1}}
\newcommand{\defiref}[1]{Definition~\ref{#1}}

\newcommand{\secref}[1]{Section~\ref{#1}}
\newcommand{\subref}[1]{Subsection~\ref{#1}}
\makeatletter
\@namedef{subjclassname@2020}{\textup{2020} Mathematics Subject Classification}
\makeatother
\DeclareMathOperator{\im}{Im}

\numberwithin{equation}{section}  
\begin{document}
\title{Some remarkable autonomous systems}
%\thanks{The author thanks Ahmed Sebbar for fruitful discussions.}
\author{Oumar Wone}
\address{Oumar Wone}
      \email{wone@chapman.edu}
\date{}
\begin{abstract}
We study the links of the Darboux-Halphen-Ramanujan system, with contact geometry, canonical coordinates of some $3$-dimensional Frobenius manifolds and projective connections on Riemann surfaces. One of our important goals is to highlight the role of contact geometry in this setting. We also study autonomous systems "derived" from the potential given by the Group-determinant, of any cyclic group $\Z/n\Z$, $n\geqslant3$, and the Klein group $\Z/2\Z\times\Z/2\Z$.
\end{abstract}
\maketitle 
\tableofcontents
\section{Introduction}
The system of Darboux-Halphen was first exhibited by Darboux \cite{darboux} in his study of triply orthogonal systems, that is a system of three one-parameter family of surfaces in $\R^3$ such that at any common point to three representatives of each family, the tangent planes of the surfaces are $2$ by $2$ orthogonal. Their solutions were given in terms of (logarithmic) derivatives of theta functions by \cite{brioschi,halphen1}. One of the remarkable works in the field of Darboux-like autonomous polynomials differential systems were made in \cite{ohyama2}. This work, in particular associates to any second order linear ODE with regular or irregular singular points on the projective line $\C P^1$, a generalized Darboux-Halphen system. As we will see in  \secref{dbh} below, Darboux-Halphen systems are related to the Chazy III equation which was first introduced by \cite{chazy}, in his study of third order ODEs solved in the highest derivative, which are autonomous and rational in the other derivatives, and which possess a Painlevé like property, of having a movable natural boundary, and whose integral curves are uniform. Chazy \cite{chazy} found thirteen such equations in total. We refer to \cite{clarkson1996} for a nice analysis regarding the Chazy equations. 

On the one hand the differential system of Ramanujan was invented by Ramanujan \cite{ramanujan1916} in his study of differential properties of (quasi) modular forms for the modular group $\rm{SL}(2,\Z)$. It implies that the three normalized Eisenstein series $E_2$, $E_4$, $E_6$ generate a differential ring over $\C$,  see \secref{dbh} below. The Ramanujan system has important arithmetical applications, see \cite{vanderpol} for instance. It was also used in (for instance) \cite{mahler} to derive transcendence type results for the field generated by $q=e^{2i\pi\tau}$, $E_2(q)$, $E_4(q)$ and $E_6(q)$, $\tau\in\mathfrak H=\{\rm{Im}(\tau)>0\}$. Besides there has been intensive investigations of systems of generalized Darboux-Ramanujan-Halphen type in the context of Gauss-Manin connections on certain moduli stacks, classifying Abelian schemes with a preferred choice of frame on their first de Rham cohomology group: concretely the higher Darboux-Ramanujan-Halphen systems (or vector fields) arise as the unique vector fields satisfying some well-chosen contraction properties, with respect to the Gauss-Manin connection, and the chosen preferred basis. We refer to \cite{fonseca1, fonseca2, movasati1, movasati2, movasati3} for an exposition of this circle of ideas.

On the other hand Group determinants were at the forefront of the invention of representation theory (character theory), \cite{conrad1998}. They were introduced by Frobenius in his attempt to answer to a question of Dedekind, see \cite{conrad1998} and references cited in there. They exist for any finite group $G$, and enjoy very nice factorization property. Below we shall only need them in the case of finite Abelian groups $G$.

The spirit of this article is the geometric study of polynomial complex autonomous systems including, but not limited to, the Darboux-Halphen-Ramanujan systems, and also autonomous systems "derived" from the potential given by the Group determinant, in the case of a finite Abelian group. 

More precisely we begin in \secref{dbh} by introducing the notations, then in \subref{contactgeo} we give a contact geometric interpretation of the Ramanujan system, see \theoref{d11}. Contact geometry is the study of (complex) manifolds $M$ of dimension $2m+1$, $m\geqslant0$, endowed with a contact structure, see \defiref{d6}, i.e. a holomorphic line subbundle $L$ of the holomorphic tangent bundle of $M$ such that if $\omega$ is a local generator of the sheaf of sections $\mathcal{O}(L)$ then $\omega\wedge(d\omega)^m\not=0.$

%Further in \subref{autonomous} we give a link between autonomous systems of ODEs and Hamiltonians systems and deduce from the Poisson theorem a way to build more first integrals of the given autonomous systems starting from a known one, see \propref{appell}. This \subref{autonomous} is a natural continuation of the Hamiltonian/contact geometry study of the Ramanujan system. 
Then we give in \theoref{canonicalcoor} the canonical coordinates of the three dimensional Frobenius manifold associated to the solution $\frac{\pi i}{3}E_2$ of the Chazy equation \eqref{d_{33}4}, in terms of its flat coordinates and some solutions of the Darboux-Halphen system \eqref{d100000}. We also give the formula of change of basis from the flat coordinate vector fields to the canonical coordinate vector fields in \propref{chang}. 

Our approach is different from the one followed in, for instance \cite{morrison, movasati4}, in that in \cite{morrison}, the authors consider only the flat coordinates and do not give the canonical coordinates explicitly in terms of the flat coordinates and they are, furthermore, interested in modular Frobenius manifolds; while in \cite{movasati4}, the authors explain how the Darboux-Halphen system arises from a $3$-dimensional Frobenius manifold, which is not our point of view: we emphasize the link between certain three dimensional autonomous systems (including the Darboux-Halphen system) and elementary symmetric functions of one variable polynomials of degree $3$ and $4$. See more precisely \lemref{gendarboix} and \lemref{gendarboix1}.

Following this we study in \secref{seckasner} some autonomous systems "derived" from the potential given by the group determinant, in the case of cyclic groups $\Z/n\Z$, $\geqslant3$ and the Klein group $\Z/2\Z\times\Z/2\Z$. Our inspiration in this \secref{seckasner} arose from reading \cite{kasner} (see also \cite{kinyon, ozekes}), where the author while looking for some particular solutions to Einstein's gravitation equations, considered the following system
\begingroup % or {
\setlength{\abovedisplayskip}{0pt}
\setlength{\belowdisplayskip}{2pt}
\setlength{\belowdisplayshortskip}{-2pt}
\begin{equation*}
\begin{split}
&x_1^\prime=x_2x_3-x_1^2\\
&x_2^\prime=x_3x_1-x_2^2\\
&x_3^\prime=x_1x_2-x_3^2,\quad^\prime=\dfrac{d}{dt}.
\end{split}
\end{equation*} 
\endgroup
Introducing the group determinant of the cyclic group $\Z/3\Z$ given by $F(\Z/3\Z)(x_1,x_2,x_3):=x_1^3+x_2^3+x_3^3-3x_1x_2x_3$, one quickly sees that the system considered in \cite{kasner} is exactly 
\begingroup % or {
\setlength{\abovedisplayskip}{0pt}
\setlength{\belowdisplayskip}{2pt}
\setlength{\belowdisplayshortskip}{-2pt}
\begin{equation*}
\begin{split}
&x_1^\prime=-\dfrac{1}{3}\dfrac{\partial F(\Z/3\Z)}{\partial x_1}\\
&x_2^\prime=-\dfrac{1}{3}\dfrac{\partial F(\Z/3\Z)}{\partial x_2}\\
&x_3^\prime=-\dfrac{1}{3}\dfrac{\partial F(\Z/3\Z)}{\partial x_3},\quad^\prime=\dfrac{d}{dt}.
\end{split}
\end{equation*} 
\endgroup
This is our motivation for studying in \secref{seckasner} the more general system
\begingroup % or {
\setlength{\abovedisplayskip}{0pt}
\setlength{\belowdisplayskip}{2pt}
\setlength{\belowdisplayshortskip}{-2pt}
\begin{equation*}
x_k^\prime=-\dfrac{\partial F(\Z/n\Z)}{\partial x_k},\quad n\geqslant3,\;k=1,\ldots,n.
\end{equation*}
\endgroup
We show among others in \theoref{kasnerth1}, how to solve explicitly such a system. More precisely we prove that there is a solution preserving map (see \defiref{defkinyon} for the precise meaning) from this system to another system \eqref{kasner4} which can be "solved" explicitly. Furthermore we show that \eqref{kasner4} admits some natural quadratic first integrals, and that it admits at least two Darboux polynomials. We recall that is a polynomial  $P$ is a Darboux polynomial for a polynomial vector field $X$ (in the same indeterminates), if there exists a polynomial $Q$ in the same variables such that $X(P)=QP$. Hence system \eqref{kasner4} is significantly different from the Darboux-Halphen system \eqref{d100000} since the later is known to admit no meromorphic first integral, see \cite{strelcyn}. We also study in \secref{seckasner} a similar system, but for the case of the first non-cyclic Abelian group, i.e. the Klein group $\Z/2\Z\times\Z/2\Z$, and find its symmetry group (see \secref{seckasner} for more details).

Finally in \secref{wirtinger} after introducing affine and projective connections on Riemann surfaces, as well as the Klein bidifferential and the Wirtinger connection associated to an even $\theta$-characteristics with non-vanishing $\theta$-constant, we prove in \theoref{projconn} that $-6e_1$, $-6e_2$, $-6e_3$, with $e_1$, $e_2$, $e_3$ defined in \eqref{darboux} and \eqref{higgs}, can be interpreted as the Wirtinger connections associated respectively to the even $\theta$-characteristics $ \begin{bmatrix}
  0\\
0
 \end{bmatrix},\begin{bmatrix}
  1\\
0
 \end{bmatrix},\begin{bmatrix}
  0\\
1
 \end{bmatrix}
$, (genus one case). Thus $e_1$, $e_2$, $e_3$, as given in \eqref{darboux} and \eqref{higgs}, in a certain sense  play a role similar to the fact that $\dfrac{\pi i}{3}E_2(\tau)$, may be seen as a projective connection on the modular curve $\mathfrak H/\rm{SL}(2,\Z)$, minus some orbifold points. \section{Darboux-Halphen-Ramanujan systems and autonomous systems}
\label{dbh}
Let us first recall some notions in order to fix notations. We remind that the normalized Eisenstein series are given by
 \begin{equation}
 \label{atiyahmacdonald}
 \begin{split}
 E_4(\tau)&=1+ 240\sum_{n=1}^{\infty} \frac{n^3 q^n}{1- q^n}=1+240\sum_{n=1}^\infty\sigma_{3}(n)q^n,\\
 E_6(\tau)&=1- 504\sum_{n=1}^{\infty} \frac{n^5 q^n}{1- q^n}=1-504\sum_{n=1}^{\infty} \sigma_5(n)q^n,\\
 E_2(\tau)&=1-24\sum_{n=1}^{\infty} \frac{n q^n}{1- q^n}=1-24\sum_{n\geqslant1}\sigma_1(n)q^n,
 \end{split}
 \end{equation} 
 $\,q=e^{2i\pi\tau},\, \tau\in\mathfrak H:=\{\im(\tau)>0\}$ and $\sigma_k(n)=\sum_{d|n}d^k$. For $\omega_1$, $\omega_2\in\C$ with $\frac{\omega_2}{\omega_1}\in\mathfrak{H}$, we set
 $ \displaystyle{ \omega_3=  -( \omega_1+ \omega_2})$ and adopt the classical notations for Weierstrass elliptic functions 
 \begin{equation}
\displaystyle \wp(u)= \wp(u; \omega_1, \omega_2)= \frac{1}{u^2}+ \sum_{m ^2+ n^{2} \neq 0} 
\left( \frac{1}{( u+ 2m \omega_1+ 2n \omega_2)^2}- \frac{1}{( 2m \omega_1+ 2n \omega_2)^2}\right).
\end{equation}
%Let us now define
%$$\zeta(u,\omega_1,\omega_3):=\dfrac{1}{u}+ \sum_{m ^2+ n^{2} \neq 0} 
%\left( \frac{1}{( u+ 2m \omega_1+ 2n \omega_3)}+ \frac{1}{( 2m \omega_1+ 2n \omega_3)}+\dfrac{u}{( 2m \omega_1+ 2n \omega_3)^2}\right)$$
Associated to the Weierstrass $\wp$ function we have the following three expressions
\begin{equation}
\label{darboux}
\begin{split}
e_1:=\wp(\omega_1;\omega_1,\omega_2),\\
e_2:=\wp(\omega_2;\omega_1,\omega_2),\\
e_3:=\wp(\omega_3;\omega_1,\omega_2),
\end{split}
\end{equation}
with $\omega_3=-(\omega_1+\omega_2)$. We further introduce another set of functions which depend on the modular variable $\tau:=\omega_2/\omega_1\in\mathfrak H$
\begin{equation}
\label{higgs}
\begin{split}
e_k=:\dfrac{1}{(2\omega_1)^2}e_k(\tau)\\
k=1,2,3.
\end{split}
\end{equation}
We will use the functions $e_k(\tau)$ in the \theoref{projconn} below.

%Let us introduce the Eisenstein series
%  \begin{equation}
%  \begin{split}
% g_2&= 60 \sum_{m^2+ n^2 \neq 0} \frac{1}{( 2m \omega_1+ 2n \omega_2)^4},\\
% g_3&= 140\sum_{m^2+ n^2 \neq 0} \frac{1}{( 2m \omega_1+ 2n \omega_2)^6}.
% \end{split}
 %\end{equation}
%They are related to the normalized Eisenstein series $E_4$ and $E_6$ by the following relations \cite{serre1970,zag2004,dolgachev}
 %\begin{equation}
% \begin{split}
% E_4(\tau)=& 12 \left( \frac{\omega_1}{\pi}\right)^4 g_2= 1+ 240\sum_{n=1}^{\infty} \frac{n^3 q^n}{1- q^n},\\
% E_6(\tau)=& 216 \left( \frac{\omega_1}{\pi}\right)^6 g_3= 1- 504\sum_{n=1}^{\infty} \frac{n^5 q^n}{1- q^n}.
% \end{split}
% \end{equation}
% Finally we have (see for ex. \cite[p. 84]{serre1970}) that the Weierstrass function $\wp$ satisfies the fundamental differential equation
% $$\wp^{\prime2}=4\wp^3-g_2\wp-g_3$$
% so that 
 %$$\wp^{\prime\prime}=6\wp^2-\dfrac{1}{2}g_2.$$
 The Eisenstein series $E_4$ and $E_6$ above, \cite{zag2004} are modular forms of weight $4$ and $6$ respectively for the group $SL(2,\Z)$. Furthermore every modular form for $SL(2,\Z)$ is uniquely expressible as a polynomial in $E_4$ and $E_6$ and the extension ring ${\C}[E_2,E_4,E_6]$ of ${\C}[E_4,E_6]$ is a differential ring. More precisely the following basic relations of Ramanujan hold \cite{zag2004,ramanujan1916}:
\begin{equation} \label{S1}
 \begin{split}
\frac{1}{2i\pi}\frac{d}{d\tau}E_4&= \frac{1}{3}(E_2 E_4- E_6), \\
\frac{1}{2i\pi}\frac{d}{d\tau}E_6&= \frac{1}{2}(E_2 E_6- E_4^2),\\
\frac{1}{2i\pi}\frac{d}{d\tau}E_2&= \frac{1}{12}(E_2^2-E_4).
\end{split}
\end{equation}
In other words, the field $ \left(\C \left( E_2, E_4, E_6\right), \dfrac{1}{2i\pi}\dfrac{d}{d\tau} \right) $ is a differential field. The subfield of constants is the field of complex numbers $\C$ (as it is embedded in the field of meromorphic functions on $\mathfrak{H}$). 
\subsection{Contact geometric interpretation of the Darboux-Ramanujan system}
\label{contactgeo}
In this subsection we introduce a slight change of notation for the Eisenstein series \eqref{atiyahmacdonald}, in order to simplify the notation in our formulas below. For $\mathfrak{H}$ the upper-half plane and $\tau\in \mathfrak{H}$, we set as before $q=e^{2i\pi\tau}$ and consider the Eisenstein series 
\begin{equation}
\label{d1}
\begin{split}
\displaystyle \mathscr{X}&=1-24\sum_{n=1}^\infty\sigma_1(n)q^n,\\
\displaystyle \mathscr{Y}&=1+240\sum_{n=1}^\infty\sigma_3(n)q^n,\\
\displaystyle \mathscr{Z}&=1-504\sum_{n=1}^\infty\sigma_5(n)q^n,
\end{split}
\end{equation}
where $\sigma_k(n)=\displaystyle \sum_{d|n}d^k$. We recall that $\mathscr X$, $\mathscr Y$, $\mathscr Z$ satisfy the following system of differential equations, hereafter called the Darboux-Ramanujan system (see equation \eqref{S1})
\begin{equation}
\label{d2}
\begin{split}
\displaystyle\dfrac{d\mathscr{X}}{d\mathfrak t}&=\dfrac{1}{12}(\mathscr{X}^2-\mathscr{Y}),\\
\dfrac{d\mathscr{Y}}{d\mathfrak t}&=\dfrac{1}{3}(\mathscr{X}\mathscr{Y}-\mathscr{Z}),\\
\dfrac{d\mathscr{Z}}{d\mathfrak t}&=\dfrac{1}{2}(\mathscr{X}\mathscr{Z}-\mathscr{Y}^2),
\end{split}
\end{equation} 
with $\mathfrak t=2i\pi\tau$. Let us give a contact geometric interpretation of the system \eqref{d2}. We start with some preliminaries.
\begin{definition}[\cite{bryant}]
\label{d6}
Let $M$ be a complex manifold of dimension $2m+1$, $T^\star M$ its holomorphic cotangent bundle and $\Omega_M^1$ the sheaf of sections of $T^\star M$. A contact structure on $M$ is a holomorphic line-subbundle $L\subset T^\star M$ such that if $\omega$ is a local generator of the sheaf of sections of $L$, $\mathcal{O}(L)\subset\Omega_M^1$, then
$$\omega\wedge(d\omega)^m­\not=0.$$
In other words there is an open covering of $M$ by open sets $\mathcal{O}_\alpha$ such that:
\begin{enumerate}
\item On each $\mathcal{O}_\alpha$ there is a holomorphic one form $\omega_\alpha$ satisfying
$$\omega_\alpha\wedge(d\omega_\alpha)^m\neq­0.$$
\item On $\mathcal{O}_\alpha\cap\mathcal{O}_\beta­\neq\emptyset$ there is defined a non-vanishing holomorphic function $f_{\alpha\beta}$ such that $\omega_\alpha=f_{\alpha\beta}\omega_\beta$, and
$$f_{\alpha\beta}f_{\beta\gamma}f_{\gamma\alpha}=1,$$
on non-empty triple intersections.
\end{enumerate}
Thus the condition $(1)$ is clearly independent of the chosen generator $\omega$.
\end{definition}
\begin{definition}
\label{d7}
A holomorphic symplectic manifold is a complex manifold $\BF$ of dimension $2n$ with a closed non-degenerate holomorphic two-form $\Omega$.

% (of type $(2,0)$): $\Omega$ establishes by non-degeneracy, an isomorphism between the holomorphic tangent bundle $T\BF$ of $\BF$ and the holomorphic cotangent bundle $T^\star \BF$ of $\BF$:
%\begin{equation*}
%\begin{split}
%\mathcal W\colon&T\BF\DistTo T^\star \BF\\
%&X\mapsto i_{X}\Omega\colon Y\mapsto\Omega(X,Y).
%\end{split}
%\end{equation*}
\end{definition}
\begin{definition}
\label{d8}
Let $\mathscr{F}\colon\BF\to\C$ a holomorphic function on the symplectic manifold $\BF$. Then the Hamiltonian vector field associated to $\mathscr{F}$ is the holomorphic vector field defined by:
$$i_{X_\mathscr{F}}\Omega=d\mathscr{F}.$$
\end{definition}
%\begin{remark}
%\label{d9}
%On a compact connected complex symplectic manifold, the only hamiltonian are the constant ones. But as in our applications only non compact complex manifolds will intervene, we will for simplicity stick with this definition. In general one should consider maps from $\BF$ to the complex projective space $\mathbb P^1_\C$.
%\end{remark}
%\begin{definition}[\cite{geiges}]
%\label{d10}
%A Liouville vector field $X$ on a complex symplectic manifold $\BF$ is a holomorphic vector field satisfying the equation
%\begin{equation}
%\label{d3}
%\mathcal{L}_X\Omega=\Omega.
%\end{equation}
%In this case, the one form $\alpha:=i_X\Omega:=\Omega(X,.)$ is a contact form on any hypersurface $W$ of $\BF$ transverse to $X$, i.e. with $X$ nowhere tangent to $W$. Such hypersurfaces are called hypersurfaces of contact type.
%\end{definition}
Let us consider now $\BF=\C^3\times\C^\times$ with coordinates $(q_1,q_2,p_1,p_2)$ and the symplectic form $\Omega=dp_1\wedge dq_1+dp_2\wedge dq_2$ together with the associated Liouville form $\alpha=p_1dq_1+p_2dq_2$. We introduce the following modification of the system \eqref{d2} by making the change of variables $\mathscr{X}_1=\dfrac{1}{6}\mathscr{X}$, $\mathscr{Y}_1=\dfrac{1}{3}\mathscr{Y}$, $\mathscr{Z}_1=\dfrac{1}{27}\mathscr{Z}$. The new system is
\begin{equation}
\label{d4}
\begin{split}
\dot{\mathscr{X}}_1&=\dfrac{1}{2}\mathscr{X}^2_1-\dfrac{1}{24}\mathscr{Y}_1,\\
\dot{\mathscr{Y}}_1&=2\mathscr{X}_1\mathscr{Y}_1-3\mathscr{Z}_1,\\
\dot{\mathscr{Z}}_1&=3\mathscr{X}_1\mathscr{Z}_1-\dfrac{1}{6}\mathscr{Y}^2_1,\quad\dot{}=\dfrac{d}{d\mathfrak t}.
%&\dfrac{d}{d\tau}=\dfrac{d\mathfrak t}{d\tau}\dfrac{d}{d\mathfrak t}=2i\pi\dfrac{d}{d\mathfrak t}
\end{split}
\end{equation}
\begin{theorem}
\label{d11}
Consider the symplectic $(M=\C^3\times\C^\times,\Omega)$ with coordinates $(q_1,q_2,p_1,p_2)$. Let $\mathscr F$ be the following Hamiltonian on $M$ given by
$$\mathscr{F}=\dfrac{1}{2}q_1^2p_1-\dfrac{1}{48}p_1^2p_2^8+3q_1q_2p_2.$$ Set $q_1=\mathscr{X}_1$, $q_2=\mathscr{Z}_1$, $p_1=\lambda \mathscr{Y}_1$ and $p_2=\lambda$. Then the Hamiltonian equations associated to $\mathscr{F}$, derived from \defiref{d8} are
\begin{equation}
\label{d5}
\begin{split}
\dot{\mathscr{X}}_1&=\dfrac{1}{2}\mathscr{X}^2_1-\dfrac{\lambda^9}{24}\mathscr{Y}_1\\
\lambda\dot{\mathscr{Y}}_1&=2\lambda \mathscr{X}_1\mathscr{Y}_1-3\lambda \mathscr{Z}_1\\
\dot{\mathscr{Z}}_1&=-\dfrac{1}{6}\lambda^9\mathscr{Y}^2_1+3\mathscr{X}_1\mathscr{Z}_1\\
\dot{\lambda}&=-3\lambda \mathscr{X}_1.
\end{split}
\end{equation}
Moreover the projection of the system \eqref{d5} onto the hypersurface $\lambda=1$ transverse to the vector field $Y=\partial_{\lambda}$ gives the system \eqref{d4}. Furthermore the restriction of the Liouville form on $\lambda=1$, is a contact $1$-form.
\end{theorem}
\begin{proof}
The system \eqref{d5} is easily deduced from the definition of the hamiltonian equations. 
%The fact that $X=p_1\partial_{p_1}+p_{2}\partial_{p_2}$ is a Liouville vector field is standard. Indeed by the Cartan formula and by the closeness of the symplectic form $\Omega$, we have 
%$$\mathcal L_X\Omega=i_Xd(\Omega)+di_X(\Omega)=di_Y(\Omega)=\Omega.$$
To prove the next assertion we remark that giving the Hamiltonian system \eqref{d5} is equivalent to requiring that $\mathscr{X}_1$, $\mathscr{Y}_1$, $\mathscr{Z}_1$ and $\lambda$ be a flow of the vector field on $M$
$$X_{\mathscr{F}}:=\left(\dfrac{1}{2}\mathscr{X}_1^2-\lambda^9\mathscr{Y}_1\right)\frac{\partial}{\partial \mathscr{X}_1}+(2\mathscr{X}_1\mathscr{Y}_1-3\mathscr{Z}_1)\frac{\partial}{\partial \mathscr{Y}_1}+\left(-\dfrac{1}{6}\lambda^9\mathscr{Y}_1^2+3\mathscr{X}_1\mathscr{Z}_1\right)\frac{\partial}{\partial \mathscr{Z}_1}-3\lambda \mathscr{X}_1\frac{\partial}{\partial \lambda}.$$
The projection of the vector field $X_\mathscr{F}$ onto the hypersurface $p_2=\lambda\equiv1$ gives the vector field
$$Y_\mathscr{F}:=\left(\dfrac{1}{2}\mathscr{X}_1^2-\mathscr{Y}_1\right)\frac{\partial}{\partial \mathscr{X}_1}+(2\mathscr{X}_1\mathscr{Y}_1-3\mathscr{Z}_1)\frac{\partial}{\partial \mathscr{Y}_1}+\left(-\dfrac{1}{6}\mathscr{Y}_1^2+3\mathscr{X}_1\mathscr{Z}_1\right)\frac{\partial}{\partial \mathscr{Z}_1}$$
whose flow on $\mathfrak H$, gives the system \eqref{d4}. Finally one easily verifies that $\alpha:=p_1dq_1+p_2dq_2$ restricts to a contact form on $\lambda=1$.
\end{proof}
\subsection{Elementary symmetric functions and the Darboux-Halphen system}
\begin{lemma}
\label{gendarboix}
Let $\gamma$ be a holomorphic function on the upper half plane such that
$$\gamma^{\prime\prime\prime}=R(\gamma,\gamma^\prime,\gamma^{\prime\prime})$$
with $R$ a polynomial with constant coefficients in $\C$. Let $$P(\mathscr{N})=\mathscr{N}^3+a_1\gamma(\tau)\mathscr{N}^{2}+a_2\gamma^\prime(\tau)\mathscr{N}+a_3\gamma^{\prime\prime}(\tau)$$ be a polynomial in $\mathscr{N}$ with $(a_i)_{1\leqslant i\leqslant 3}\in\C^\times$ and such that the discriminant of $P(\mathscr{N})$ does not vanish on $\mathfrak{H}$. Then the roots of $P(\mathscr{N})$ satisfy an autonomous polynomial differential system. 
\end{lemma}
\begin{proof}
The elementary symmetric functions of the roots of $P(\mathscr{N})$ are expressed in terms of $a_1\gamma(\tau)$, $a_2\gamma^\prime(\tau)$, $a_3\gamma^{\prime\prime}(\tau)$ as follows
\begin{equation}
\label{localsys}
\begin{split}
&\mathscr{N}_1+\mathscr{N}_2+\mathscr{N}_3=-a_1\gamma(\tau),\\
&\mathscr{N}_1\mathscr{N}_2+\mathscr{N}_1\mathscr{N}_3+\mathscr{N}_2\mathscr{N}_3=a_2\gamma^\prime(\tau),\\
&\mathscr{N}_1\mathscr{N}_2\mathscr{N}_3=-a_3\gamma^{\prime\prime}(\tau).
\end{split}
\end{equation}
 Taking the derivation with respect to $\tau$ of the gotten equations and expressing for instance $\dfrac{d\mathscr{N}_1}{d\tau}$, in terms of $\dfrac{d\mathscr{N}_2}{d\tau}$, $\dfrac{d\mathscr{N}_3}{d\tau}$ and $a_1\gamma(\tau)$ via the derivation of the first identity in \eqref{localsys}, and similarly for the second and third equation of \eqref{localsys}, one obtains a linear non-homogeneous system for the $\left(\dfrac{d\mathscr{N}_i}{d\tau}\right)_{1\leqslant i\leqslant 3}$. Indeed we have the following steps
 $$\frac{d}{d\tau}\mathscr{N}_1=-\frac{d}{d\tau}\mathscr{N}_2-\frac{d}{d\tau}\mathscr{N}_3-a_1\gamma^{\prime}(\tau)$$
\begin{equation*}
\begin{split}
&\left(-\frac{d}{d\tau}\mathscr{N}_2-\frac{d}{d\tau}\mathscr{N}_3-a_1\gamma^{\prime}(\tau)\right)(\mathscr{N}_2+\mathscr{N}_3)+\frac{d}{d\tau}(\mathscr{N}_2)\mathscr{N}_1+\frac{d}{d\tau}(\mathscr{N}_3)\mathscr{N}_1\\
&+\frac{d}{d\tau}(\mathscr{N}_2)\mathscr{N}_3+\frac{d}{d\tau}(\mathscr{N}_3)\mathscr{N}_2=a_2\gamma^{\prime\prime}(\tau)
\end{split}
\end{equation*}
and
$$\left(-\frac{d}{d\tau}(\mathscr{N}_2)-\frac{d}{d\tau}(\mathscr{N}_3)-a_1\gamma^{\prime}(\tau)\right)\mathscr{N}_2\mathscr{N}_3+\frac{d}{d\tau}(\mathscr{N}_2)\mathscr{N}_1\mathscr{N}_3+\frac{d}{d\tau}(\mathscr{N}_3)\mathscr{N}_1\mathscr{N}_2=-a_3\gamma^{\prime\prime\prime}(\tau).$$
This gives the non-homogeneous system for the $\dfrac{d\mathscr{N}_i}{d\tau}$
\begin{equation}
\label{delignemumfordstacks}
\begin{split}
\frac{d}{d\tau}\mathscr{N}_1&=-\frac{d}{d\tau}\mathscr{N}_2-\frac{d}{d\tau}\mathscr{N}_3-a_1\gamma^\prime\\
\frac{d}{d\tau}(\mathscr{N}_2)(\mathscr{N}_1-\mathscr{N}_2)+\frac{d}{d\tau}(\mathscr{N}_3)(\mathscr{N}_1-\mathscr{N}_3)&=a_2\gamma^{\prime\prime}+a_1\gamma^\prime(\mathscr{N}_2+\mathscr{N}_3)\\
\frac{d}{d\tau}(\mathscr{N}_2)\mathscr{N}_3(\mathscr{N}_1-\mathscr{N}_2)+\frac{d}{d\tau}(\mathscr{N}_3)\mathscr{N}_2(\mathscr{N}_1-\mathscr{N}_3)&=-a_3\gamma^{\prime\prime\prime}+a_1\gamma^\prime\mathscr{N}_2\mathscr{N}_3.
\end{split}
\end{equation}
The determinant of the system in the $\dfrac{d\mathscr{N}_i}{d\tau}$, \eqref{delignemumfordstacks} is 
$$\mathscr D:=\left|\begin{array}{ccc}1 & 1 & 1 \\\mathscr{N}_1-\mathscr{N}_2 & \mathscr{N}_1-\mathscr{N}_3 & 0 \\\mathscr{N}_3(\mathscr{N}_1-\mathscr{N}_2) & \mathscr{N}_2(\mathscr{N}_1-\mathscr{N}_3) & 0\end{array}\right|=(\mathscr{N}_1-\mathscr{N}_2)(\mathscr{N}_1-\mathscr{N}_3)(\mathscr{N}_2-\mathscr{N}_3)$$
the square of which is given up to sign by the discriminant of the polynomial $P(\mathscr{N})$. By hypothesis this discriminant does not vanish of $\mathfrak{H}$. This enables us to apply Cramer's rule. Finally the condition $\gamma^{\prime\prime\prime}=R(\gamma,\gamma^\prime,\gamma^{\prime\prime})$ forces the $\mathscr{N}_i,\,1\leqslant i\leqslant 3$ to satisfy an autonomous polynomial differential system.
\end{proof}
Similarly
\begin{lemma}
\label{gendarboix1}
If $\gamma$ is a holomorphic function on the upper half plane such that
$$\gamma^{\prime\prime\prime\prime}=R(\gamma,\gamma^\prime,\gamma^{\prime\prime},\gamma^{\prime\prime\prime})$$
with $R$ a polynomial with constant coefficients in $\C$. Let $$P(\mathscr{N})=\mathscr{N}^4+a_1\gamma(\tau)\mathscr{N}^{3}+a_2\gamma^\prime(\tau)\mathscr{N}^2+a_3\gamma^{\prime\prime}(\tau)\mathscr{N}+a_4\gamma^{\prime\prime}$$ be a polynomial in $\mathscr{N}$ with $(a_i)_{1\leqslant i\leqslant 4}\in\C^\times$ and such that the discriminant of $P(\mathscr{N})$ does not vanish on $\mathfrak{H}$. Then the roots of $P(\mathscr{N})$ satisfy an autonomous polynomial differential system. 
\end{lemma}
\begin{corollary}[\cite{Dubrovin}]
If $\gamma$ is the solution of the Chazy equation,
$$\gamma^{\prime\prime\prime}= 6 \gamma \gamma^{\prime\prime} - 9{\gamma^{\prime}}^2,\quad \gamma=
\gamma(\tau)$$
given by $\gamma=\dfrac{\pi i}{3}E_2$. Then the three solutions $\mathscr{N}_1(\tau)$, $\mathscr{N}_2(\tau)$, $\mathscr{N}_3(\tau)$ of the cubic equation
 \begin{equation}
 \label{eqr1002} 
  \mathscr{N}^3+ \frac{3}{2}
\gamma(\tau) \mathscr{N}^2 +\frac{3}{2}\gamma^\prime(\tau) \mathscr{N} + \frac{1}{4}
\gamma^{\prime\prime}(\tau) =0
\end{equation}
 are solutions of the Darboux-Halphen system:
\begin{equation}
\label{d100000}
\begin{split}
 \frac{d}{d\tau}\mathscr{N}_1&= -\mathscr{N}_1( \mathscr{N}_2+ \mathscr{N}_3) + \mathscr{N}_2\mathscr{N}_3\\
\frac{d}{d\tau}\mathscr{N}_2&= -\mathscr{N}_2( \mathscr{N}_1+ \mathscr{N}_3) + \mathscr{N}_1\mathscr{N}_3\\
 \frac{d}{d\tau}\mathscr{N}_3&= -\mathscr{N}_3( \mathscr{N}_1+ \mathscr{N}_2) + \mathscr{N}_1\mathscr{N}_2.
\end{split}
\end{equation}
%which gives the equation \eqref{S4} by setting $X_i= -2\mathscr{N} _i ,\,i=1,2,3$ .
\end{corollary}
\begin{proof}
This follows from \lemref{gendarboix} by taking $a_1=3/2$, $a_2=3/2$, $a_3=1/4$, and by using, \cite{mahler}, the algebraic independence of $\gamma, \gamma^\prime, \gamma^{\prime\prime}$ over $\C$, $\gamma=\dfrac{\pi i}{3}E_2$ (the discriminant of \eqref{eqr1002} is a polynomial over $\C$ in $\gamma, \gamma^\prime, \gamma^{\prime\prime}$). This allows us to apply Cramer's rule and obtain for the value of $\dfrac{d\mathscr{N}_2}{d\tau}$ the following
$$\frac{d}{d\tau}\mathscr{N}_2=\dfrac{\dfrac{1}{4}\gamma^{\prime\prime\prime}(\mathscr{N}_1-\mathscr{N}_3)+\dfrac{3}{2}\gamma^{\prime\prime}\mathscr{N}_2(\mathscr{N}_1-\mathscr{N}_3)+\dfrac{3}{2}\gamma^\prime\mathscr{N}_2^2(\mathscr{N}_1-\mathscr{N}_3)}{\mathscr{D}},$$
with 
%\mathscr D=(\mathscr{N}_1 − \mathscr{N}_2)(\mathscr{N}_1 − \mathscr{N}_3)(\mathscr{N}_2 − \mathscr{N}_3)
$$\mathscr{D}=(\mathscr N_1-\mathscr N_2)(\mathscr N_1-\mathscr N_3)(\mathscr N_2-\mathscr N_3).$$
Using the fact that $\gamma$ satisfies the Chazy equation
$$ \gamma^{\prime\prime\prime}= 6 \gamma \gamma^{\prime\prime} - 9{\gamma^{\prime}}^2,\quad \gamma=
\gamma(\tau)$$
and the expressions of $\gamma^\prime$ and $\gamma^{\prime\prime}$ in terms of $\mathscr{N}_1$, $\mathscr{N}_2$ and $\mathscr{N}_3$, one gets after a short computation the value of $\dfrac{d}{d\tau}\mathscr{N}_2$:
$$\dfrac{d}{d\tau}\mathscr{N}_2= -\mathscr{N}_2( \mathscr{N}_1+ \mathscr{N}_3) + \mathscr{N}_1\mathscr{N}_3.$$
The determination of the derivatives $\dfrac{d}{d\tau}\mathscr{N}_1$ and $\dfrac{d}{d\tau}\mathscr{N}_3$ follows in the same manner.
\end{proof}

\begin{remark}
If $a_3\mathscr N^3+a_2\mathscr N^2+a_1\mathscr N+a_0$ is a general polynomial of degree $3$ in the independent variable $\mathscr N$, then its roots as function of the parameters $a_0$, $a_1$, $a_2$, $a_3$: $X(a_0,a_1,a_2,a_3)$, satisfy a system of PDEs called the GKZ-system (where GKZ stands for Gelfand-Kapranov-Zelevinsky). We refer to \cite{sturmfels}. Furthermore to our knowledge, the \lemref{gendarboix1} and \lemref{gendarboix} do not generalize.
\end{remark}
\section{WDVV equations and Chazy equation: canonical coordinates}
\label{wdvv}
In this section we determine the canonical coordinates of some Frobenius manifolds associated to the solution $\dfrac{\pi i}{3}E_2$ of the Chazy equation 
$$\gamma^{\prime\prime\prime}=6\gamma\gamma^{\prime\prime}-9\gamma^{\prime2}$$
in terms of its flat coordinates and some solutions to the Darboux-Halphen system \eqref{d100000}. We adopt the notations of \cite[chap.~1, appen. C]{Dubrovin}. 

If one looks for a holomorphic function $F(t_1,t_2,t_3)$ on a domain of $\C^3$ (not fixed for the moment) of the form 
$$F(t_1,t_2,t_3)=\frac{1}{2}(t_1)^2t_3+\frac{1}{2}t_1(t_2)^2+f(t_2,t_3)$$
such that its third derivatives 
$$C_{ijk}:=\dfrac{\partial F(t)}{\partial t_i\partial t_j\partial t_k}$$
satisfy the following $3$ conditions
\begin{enumerate}
\item Normalization:
$$\eta_{ij}:=C_{1ij}$$
is a constant non-singular matrix. We set $\eta^{ij}:=(\eta_{ij})^{-1}$.
\item Associativity:
$$C_{ij}^k(t):=\eta^{kl}C_{lij}(t)$$
(summation convention assumed) for any $t=(t_1,t_2,t_3)$, defines in the three dimensional space with basis $\mathfrak{e}_1$, $\mathfrak e_2$, $\mathfrak e_3$ a structure of a (commutative) associative algebra $A_t$ (with unity $\mathfrak e_1$): 
$$\mathfrak{e}_i\mathfrak{e}_j=C_{ij}^k(t)\mathfrak e_k,$$
where we have assumed the summation convention.
\end{enumerate}
Then one finds the table
\begin{equation}
\label{d1800}
\begin{split}
&\mathfrak e_2^2=f_{t_2t_3t_3}\mathfrak e_1+f_{t_2t_2t_2}\mathfrak e_2+\mathfrak e_3\\
&\mathfrak e_2\mathfrak e_3=f_{t_2t_3t_3}\mathfrak e_1+f_{t_2t_2t_3}\mathfrak e_2\\
&\mathfrak e_3^2=f_{t_3t_3t_3}\mathfrak e_1+f_{t_2t_3t_3}\mathfrak e_2
\end{split}
\end{equation}
and that the two associativity conditions
$$(\mathfrak e_2^2)\mathfrak e_3=\mathfrak e_2(\mathfrak e_2\mathfrak e_3),\quad\quad(\mathfrak e_3^2)\mathfrak e_2=\mathfrak e_3(\mathfrak e_3\mathfrak e_2)$$
are equivalent to 
\begin{equation}
\label{d1900}
f_{t_2t_2t_3}^2=f_{t_3t_3t_3}+f_{t_2t_3t_3}f_{t_2t_2t_2}.
\end{equation}
If moreover one requires that\\
$(3)$ $F$ be quasihomogeneous
$$F(ct_1,c^{1/2}t_2,t_3)=c^2F(t_1,t_2,t_3)$$
for $c\in\C^\times$ (modulo quadratic terms) and furthermore that $F$ be periodic of period $1$ in its third variable $t_3$ (modulo quadratic terms), and analytic at the point $t_1=t_2=0,\;\;t_3=i\infty$, then one finds that $F$ must have the form
\begin{equation}
\label{d_{33}1}
F=\frac{1}{2}(t_1)^2t_3+\frac{1}{2}t_1(t_2)^2-\frac{(t_2)^4}{16}\gamma(t_3)
\end{equation}
for some unknown function $\gamma(\tau)$ analytic at $\tau=i\infty$
\begin{equation}
\label{d_{33}2}
\gamma(\tau)=\displaystyle\sum_{m\geqslant0}b_mq^m,\;\;q=e^{2\pi i\tau}.
\end{equation}
The coefficients $b_m$ are defined up to shift
\begin{equation}
\label{d_{33}3}
\tau\mapsto\tau+\tau_0,\quad b_m\mapsto b_me^{2\pi in\tau_0},\;\tau_0\in\mathfrak H;
\end{equation}
and one has more precisely
$$\gamma(\tau)=\dfrac{\pi i}{3}E_2(\tau).$$
Moreover the function $\gamma$ must satisfy the Chazy equation 
\begin{equation}
\label{d_{33}4}
\gamma^{\prime\prime\prime}=6\gamma\gamma^{\prime\prime}-9\gamma^{\prime2}.
\end{equation}
In this situation the degrees of the flat (by definition) variables
are
\begin{equation}
\label{d_{33}5}
\deg(t_1)=1,\quad\deg(t_2)=1/2,\quad\deg(t_3)=0,
\end{equation}
 the Euler vector field is
\begin{equation}
\label{d_{33}6}
E=t_1\frac{\partial}{\partial t_1}+\frac{1}{2}t_2\frac{\partial}{\partial t_2}
\end{equation}
cf. \cite[chap.~1]{Dubrovin}: $E(F)=2F$. The associativity condition is a particular case of what is called the WDVV equations, where WDVV stands for Witten-Dijkgraaf-Verlinde-Verlinde. Any solution of the WDVV equations endows its domain of definition with a so-called structure of Frobenius manifold \cite[chap.~1]{Dubrovin}. So $F$ given in \eqref{d_{33}1} endows $\C^2\times \mathfrak H$ with a structure of Frobenius manifold. We have
\begin{theorem}
\label{canonicalcoor}
The canonical coordinates of the Frobenius manifold $\C^2\times\mathfrak H$ with flat coordinates $t_1$, $t_2$, $t_3$, of respective degrees given by \eqref{d_{33}5}, with Euler vector field given by \eqref{d_{33}6} and with potential $F$ given by \eqref{d_{33}1}
can be expressed in the form 
\begin{equation}
\label{d_{33}7}
\begin{split}
u_1&=t_1+\frac{1}{2}t_2^2\mathscr N_1(t_3),\\
u_2&=t_1+\frac{1}{2}t_2^2\mathscr N_2(t_3),\\
u_3&=t_1+\frac{1}{2}t_2^2\mathscr N_3(t_3),\\
&\text{$(t_1,t_2,t_3)\in\C\times\C^{\times}\times\mathfrak H$}.
\end{split}
\end{equation}
Here $\mathscr N_1(t_3)$, $\mathscr N_2(t_3)$, $\mathscr N_1(t_3)$ are the three roots of the cubic equation
\eqref{eqr1002}. The metric $\eta$ in flat coordinates is given by
$$\eta=\displaystyle\sum_{1\leqslant i\leqslant3,\,1\leqslant j\leqslant3}\eta_{ij}dt_i dt_j=dt_2^2+2dt_1dt_3.$$
Furthermore the time-dependent Poisson system on $\C^3$ with the standard Poisson structure corresponding to the antisymmetric matrix
$$\left(\begin{array}{ccc}0 & 1 & -1 \\-1 & 0 & 1 \\1 & -1 & 0\end{array}\right)$$
and for the Poisson Hamiltonian
$$H:=\dfrac{1}{2}\left[\dfrac{\Omega_1^2}{s-1}+\dfrac{\Omega_2^2}{s}\right],$$
and given by
\begin{equation}
\label{d_{33}_8}
\begin{split}
\frac{d\Omega_1}{ds}&=\frac{1}{s}\Omega_2\Omega_2,\\
\frac{d\Omega_2}{ds}&=-\frac{1}{s-1}\Omega_1\Omega_3,\\
\frac{d\Omega_3}{ds}&=\frac{1}{s(s-1)}\Omega_1\Omega_2,\\
s&=\frac{\mathscr N_3(t_3)-\mathscr N_1(t_3)}{\mathscr N_2(t_3)-\mathscr N_1(t_3)},
\end{split}
\end{equation}
 admits the solution
\begin{equation}
\label{d_{33}_9}
\begin{split}
\Omega_1&=\dfrac{\mathscr N_1}{2\sqrt{(\mathscr N_2-\mathscr N_1)(\mathscr N_1-\mathscr N_3)}},\\
\Omega_2&=\dfrac{\mathscr N_2}{2\sqrt{(\mathscr N_3-\mathscr N_2)(\mathscr N_2-\mathscr N_1)}},\\
\Omega_3&=\dfrac{\mathscr N_3}{2\sqrt{(\mathscr N_1-\mathscr N_3)(\mathscr N_3-\mathscr N_2)}}.
\end{split}
\end{equation}
\end{theorem}
\begin{proof}
By definition (cf. \cite[lect.~1]{Dubrovin}) we have
$$\eta_{ij}=\partial_{t_1}\partial_{t_i}\partial_{t_j}F.$$
Hence we find
\begin{equation}
\label{d_{33}10}
\begin{split}
&\eta_{12}=\eta_{21}=0,\\
&\eta_{13}=\eta_{31}=1,\\
&\eta_{23}=\eta_{32}=0,\\
&\eta_{11}=0,\;\eta_{33}=0,\;\eta_{22}=1,
\end{split}
\end{equation}
and this gives immediately the metric $\eta=dt_2^2+2dt_1dt_3$. Also from \cite[lect.~1]{Dubrovin} the structure constants of the deformed Frobenius algebra associated to the solution $F$ of the WDVV equation is given by
\begin{equation}
\label{d_{33}11}
C_{ij}^k:=\eta^{kl}C_{ijl}
\end{equation}
where $C_{ijk}=\partial_{t_i}\partial_{t_j}\partial_{t_k}F$, $\eta^{ij}=(\eta_{ij})^{-1}$ and where we adopt Einstein's summation convention. Hence for the $C_{ij}^k$ we obtain
\begin{equation}
\label{d_{33}12}
\begin{split}
&C_{11}^1=1,\;\;C_{11}^2=0,\;\;C_{11}^3=0,\\
&C_{22}^1=-(3/4)t_2^2\gamma^{\prime}(t_3),\;\;C_{22}^2=-(3/2)t_2\gamma(t_3),\;\;C_{22}^3=1,\\
&C_{12}^1=0,\;\;C_{12}^2=1,\;\;C_{12}^3=0,\\
&C_{33}^1=-\frac{t_2^4}{16}\gamma^{\prime\prime\prime}(t_3),\;\;C_{33}^2=-\frac{1}{4}\gamma^{\prime\prime}(t_3)t_2^3,\;\;C_{11}^3=0,\\
&C_{13}^1=0,\;\;C_{13}^2=0,\;\;C_{13}^3=1,\\
&C_{21}^1=0,\;\;C_{21}^2=1,\;\;C_{21}^3=0,\\
&C_{23}^1=-\frac{1}{4}\gamma^{\prime\prime}(t_3)t_2^3,\;\;C_{23}^2=-\frac{3}{4}t_2^2\gamma^\prime(t_3),\;\;C_{23}^3=0,\\
&C_{31}^1=0,\;\;C_{31}^2=0,\;\;C_{31}^3=1,\\
&C_{32}^1=-\frac{1}{4}\gamma^{\prime\prime}(t_3)t_2^3,\;\;C_{32}^2=-\frac{3}{4}t_2^2\gamma^\prime(t_3),\;\;C_{32}^3=0.
\end{split}
\end{equation}
In order to go further we need to compute the intersection form given by \cite[lect.~3]{Dubrovin} (Einstein's summation convention assumed in the following whenever there are repeated indices)
\begin{equation}
\label{d_{33}13}
g^{ij}(t):=E^l(t)C_l^{ij}(t)
\end{equation}
where
\begin{equation}
\label{d_{33}14}
C_k^{ij}=\eta^{il}C_{lk}^j
\end{equation}
and $E=(E^1,E^2,E^3)=(t_1,\frac{1}{2}t_2,0)$ is the Euler vector field written in components.
The $C_k^{ij}$ are given by the following
\begin{equation}
\label{d_{33}15}
\begin{split}
&C_1^{11}=0,\,\,C_1^{12}=0,\,\,C_1^{13}=1,\,\,C_1^{21}=0,\,\,C_1^{22}=1,\,\,C_1^{23}=0,\,\,C_1^{31}=1,\,\,C_1^{32}=0,\,\,C_1^{33}=0\\
&C_2^{11}=-\frac{1}{4}\gamma^{\prime\prime}(t_3)t_2^3,\,\,C_2^{12}=-\frac{3}{4}\gamma^{\prime}(t_3)t_2^2,\,\,C_2^{13}=0,\,\,C_2^{21}=-\frac{3}{4}\gamma^{\prime}(t_3)t_2^2,\,\,\\
&C_2^{22}=-\frac{3}{2}\gamma(t_3)t_2,\,\,C_2^{23}=1,\,\,C_2^{31}=0,\,\,C_2^{32}=1,\,\,C_2^{33}=0\\
&C_3^{11}=-\frac{t_2^4}{16}\gamma^{\prime\prime\prime}(t_3),\,\,C_3^{12}=-\frac{1}{4}\gamma^{\prime\prime}(t_3)t_2^3,\,\,C_3^{13}=0,\,\,C_3^{21}=-\frac{1}{4}\gamma^{\prime\prime}(t_3)t_2^3,\\
&C_3^{22}=-\frac{3}{4}\gamma^\prime(t_3)t_2^2,\,C_3^{23}=0,\,\,C_3^{31}=0,\,\,C_3^{32}=0,\,\,C_3^{33}=1.
\end{split}
\end{equation}
Using this we get for the intersection form
\begin{equation}
\label{d_{33}16}
\begin{split}
&g^{11}=-\frac{1}{8}\gamma^{\prime\prime}(t_3)t_2^4,\;\;g^{12}=-\frac{3}{8}t_2^3\gamma^{\prime}(t_3),\;\;g^{13}=t_1,\\
&g^{21}=-\frac{3}{8}\gamma^{\prime}(t_3)t_2^3,\;\;g^{22}=t_1-\frac{3}{4}t_2^2\gamma(t_3),\;\;g^{23}=\frac{1}{2}t_2,\\
&g^{31}=t_1,\;\;g^{32}=\frac{1}{2}t_2,\;\;g^{33}=0.
\end{split}
\end{equation}
According to \cite[prop.~3.3]{Dubrovin} if $(t_1,t_2,t_3)\in\C^2\times\mathfrak H$ is a point such that the roots of the following degree $3$ polynomial in the variable $u$ are distinct then its three roots $u_1(t_1,t_2,t_3)$,  $u_2(t_1,t_2,t_3)$, $u_3(t_1,t_2,t_3)$ constitute the canonical coordinates in a neighborhood of $(t_1,t_2,t_3)$. The polynomial in question is
\begin{equation}
\label{d_{33}17}
\det(g^{ij}(t_1,t_2,t_3)-u\eta^{ij})
\end{equation}
where 
$$\eta^{ij}=\left(\begin{array}{ccc}0 & 0 & 1 \\0 & 1 & 0 \\1 & 0 & 0\end{array}\right).$$
Using \eqref{d_{33}16} we find after developing the determinant that
\begin{equation}
\label{d_{33}18}
\begin{split}
\det(g^{ij}(t_1,t_2,t_3)-u\eta^{ij})&=u^3+(-3t_1+\frac{3}{4}\gamma(t_3)t_2^2)u^2+(3t_1^2-\frac{3}{2}t_1t_2^2\gamma(t_3)+\frac{3}{8}t_2^4\gamma^\prime(t_3))u\\
&-t_1^3+\frac{3}{4}t_1^2t_2^2\gamma(t_3)-\frac{3}{8}t_2^4t_1\gamma^\prime(t_3)+\frac{1}{32}t_2^6\gamma^{\prime\prime}(t_3).
\end{split}
\end{equation}
One easily verifies that its three roots are given by
\begin{equation*}
\begin{split}
u_1&=t_1+\frac{1}{2}t_2^2\mathscr N_1(t_3),\\
u_2&=t_1+\frac{1}{2}t_2^2\mathscr N_2(t_3),\\
u_3&=t_1+\frac{1}{2}t_2^2\mathscr N_3(t_3),\\
&\text{$(t_1,t_2,t_3)\in\C\times\C\times\mathfrak H$},
\end{split}
\end{equation*}
where $\mathscr N_1(t_3)$, $\mathscr N_2(t_3)$, $\mathscr N_3(t_3)$ are the three roots of the cubic
\eqref{eqr1002}. When $(t_1,t_2,t_3)\in\C\times \C^{\times}\times \mathfrak H$ then $u_1(t_1,t_2,t_3)$, $u_2(t_1, t_2,t_3)$ and $u_3(t_1,t_2,t_3)$ are distinct, hence they are the canonical coordinates of the underlying Frobenius manifold in that neighborhood.

One easily finds the Poisson Hamiltonian equations for the standard Poisson structure and the given Hamiltonian $H$. To see that $\Omega_1$, $\Omega_2$, $\Omega_3$ given in \eqref{d_{33}_9} satisfy \eqref{d_{33}_8} one uses \eqref{d100000}, the chain rule and the identity
\begin{equation}
\label{d_{33}19}
2\dfrac{dt_3}{ds}\times\dfrac{(\mathscr N_1-\mathscr N_3)(\mathscr N_2-\mathscr N_3)}{(\mathscr N_2-\mathscr N_1)}=1
\end{equation}
which follows from
$$s=\frac{\mathscr N_3(t_3)-\mathscr N_1(t_3)}{\mathscr N_2(t_3)-\mathscr N_1(t_3)}$$
and \eqref{eqr1002}.
\end{proof}
\begin{proposition}[Change of basis vector fields]
\label{chang}
For $(t_1,t_2,t_3)\in\C\times\C^\times\times\mathfrak H$ and if $u_1(t_1,t_2,t_3)$, $u_2(t_1,t_2,t_3)$, $u_3(t_1,t_2,t_3)$ constitute the corresponding canonical coordinates we have
\begin{equation}
\label{d48}
\left(\begin{array}{c}\frac{\partial}{\partial u_1} \\\\\frac{\partial}{\partial u_2} \\\\\frac{\partial}{\partial u_3}\end{array}\right)=\dfrac{1}{\mathcal A}\left(\begin{array}{ccc}\mathcal X & t_2^2\mathscr{N}_1(\mathscr{N}_3-\mathscr{N}_2) & t_2(\mathscr{N}_3-\mathscr{N}_2) \\\\\mathcal Y &t_2^2\mathscr{N}_2(\mathscr{N}_1-\mathscr{N}_3) & t_2(\mathscr{N}_1-\mathscr{N}_3)\\\\\mathcal Z& t_2^2\mathscr{N}_3(\mathscr{N}_2-\mathscr{N}_1) & t_2(\mathscr{N}_2-\mathscr{N}_1)\end{array}\right)\left(\begin{array}{c}\frac{\partial}{\partial t_1} \\\\\frac{\partial}{\partial t_2} \\\\\frac{\partial}{\partial t_3}\end{array}\right)\end{equation}
where 
\begin{equation}
\label{d49}
\begin{split}
&\mathcal A=t_2^3(\mathscr N_2-\mathscr N_1)(\mathscr N_1-\mathscr N_3)(\mathscr N_2-\mathscr N_3),\\
&\mathcal X=\frac{t_2^3}{2}(\mathscr{N}_2^2(\mathscr N_1-\mathscr N_3)-\mathscr{N}_3^2(\mathscr N_1-\mathscr N_2)),\\
&\mathcal Y=\frac{t_2^3}{2}(\mathscr{N}_1^2(\mathscr N_3-\mathscr N_2)-\mathscr{N}_3^2(\mathscr N_1-\mathscr N_2)),\\
&\mathcal Z=\frac{t_2^3}{2}(\mathscr{N}_1^2(\mathscr N_3-\mathscr N_2)-\mathscr{N}_2^2(\mathscr N_3-\mathscr N_2)).
\end{split}
\end{equation}
\end{proposition}
\begin{proof}
This follows straightforwardly from the chain rule 
\begin{equation}
\label{d50}
\begin{split}
&\partial_{t_1}=\displaystyle\sum_{i=1}^3\frac{\partial u_i}{\partial t_1}\partial_{u_i}=\partial_{u_1}+\partial_{u_2}+\partial_{u_3},\\
&\partial_{t_2}=\displaystyle\sum_{i=1}^3\frac{\partial u_i}{\partial t_2}\partial_{u_i}=t_2(\partial_{u_1}+\partial_{u_2}+\partial_{u_3}),\\
&\partial_{t_3}=\displaystyle\sum_{i=1}^3\frac{\partial u_i}{\partial t_3}\partial_{u_i}=\frac{t_2^2}{2}(\dot{\mathscr N_1}\partial_{u_1}+\dot{\mathscr N_2}\partial_{u_2}+\dot{\mathscr N_3}\partial_{u_3}),
\end{split}
\end{equation}
and \eqref{eqr1002}. Here we have set $\partial_{t_i}=\frac{\partial}{\partial t_i}$ for $i=1,2,3$, $\partial_{u_i}=\frac{\partial}{\partial u_i}$ for $i=1,2,3$, $^{.}=\frac{\partial}{\partial t_3}$ and we have written $\mathscr N_i$ for $i=1,2,3$, to mean $\mathscr N_i(t_3)$.
\end{proof}
Let us show another way to derive \eqref{d_{33}_8} by exhibiting its link with Schwarzian equations
\begin{proposition}
\label{newprop}
Let $\mathscr N_1(t_3)$, $\mathscr N_2(t_3)$, $\mathscr N_3(t_3)$, $t_3\in\mathscr H$ be three distinct solutions of the Darboux-Halphen system
  \begin{equation}
\label{dhna}
\begin{split}
 \frac{d}{dt_3}\mathscr{N}_1&= -\mathscr{N}_1( \mathscr{N}_2+ \mathscr{N}_3) + \mathscr{N}_2\mathscr{N}_3\\
\frac{d}{dt_3}\mathscr{N}_2&= -\mathscr{N}_2( \mathscr{N}_1+ \mathscr{N}_3) + \mathscr{N}_1\mathscr{N}_3\\
 \frac{d}{dt_3}\mathscr{N}_3&= -\mathscr{N}_3( \mathscr{N}_1+ \mathscr{N}_2) + \mathscr{N}_1\mathscr{N}_2.
\end{split}
\end{equation}
Define 
$$s=\dfrac{\mathscr{N}_1-\mathscr{N}_3}{\mathscr{N}_1-\mathscr{N}_2}.$$
Then $s$ is a solution of Schwarz's equation
\begin{equation}
\label{schwarziander}
\dfrac{s^{\prime\prime\prime}}{s^\prime}-\dfrac{3}{2}\left(\dfrac{s^{\prime\prime}}{s^\prime}\right)^2=-\dfrac{1}{2}\left(\dfrac{1}{s^2}+\dfrac{1}{(s-1)^2}-\dfrac{1}{s(s-1)}\right),\quad^\prime=\dfrac{d}{dt_3}.
\end{equation}
Conversely any solution of \eqref{schwarziander} provides a solution for the Darboux-Halphen system \eqref{dhna}. In fact any solution to \eqref{dhna}, up to $SL(2,\C)$ action, is given by the following triplet
\begin{equation}
\label{dhnaa}
\begin{split}
\mathscr N_1&=-\dfrac{1}{2}\dfrac{d}{dt_3}\log\left(\dfrac{s^\prime}{s}\right)=:-\dfrac{1}{2}\dfrac{d}{dt_3}\log v_1\\
\mathscr N_2&=-\dfrac{1}{2}\dfrac{d}{dt_3}\log\left(\dfrac{s^\prime}{s -1}\right)=:-\dfrac{1}{2}\dfrac{d}{dt_3}\log v_2\\
\mathscr N_3&=-\dfrac{1}{2}\dfrac{d}{dt_3}\log\left(\dfrac{s^\prime}{s(s-1)}\right)=:-\dfrac{1}{2}\dfrac{d}{dt_3}\log v_3.
\end{split}
\end{equation}
We introduce the ansatz $(\Delta_1,\Delta_2,\Delta_3)$ via the following differential system
\begin{equation}
\label{ansatz}
\begin{split}
\dfrac{d}{dt_3}\Delta_1&=\Delta_2\Delta_3-\Delta_1(\mathscr N_2+\mathscr N_3)\\
\dfrac{d}{dt_3}\Delta_2&=\Delta_3\Delta_1-\Delta_2(\mathscr N_3+\mathscr N_1)\\
\dfrac{d}{dt_3}\Delta_3&=\Delta_2\Delta_1-\Delta_3(\mathscr N_2+\mathscr N_1),
\end{split}
\end{equation}
and
\begin{equation}
\label{ansatz1}
\begin{split}
&w_1=\dfrac{\Delta_1}{\sqrt{v_2v_3}};\qquad w_2=\dfrac{\Delta_2}{\sqrt{v_3v_1}}\\
&w_3=\dfrac{\Delta_3}{\sqrt{v_1v_2}}.
\end{split}
\end{equation}
Then we find
\begin{equation}
\label{ansatz3}
\begin{split}
\dfrac{d w_1}{ds}&=\dfrac{w_2w_3}{s}\\
\dfrac{d w_2}{ds}&=\dfrac{w_3w_1}{1-s}\\
\dfrac{d w_3}{ds}&=\dfrac{w_2w_3}{s(s-1)}
\end{split}
\end{equation}
\end{proposition}
\begin{proof}
Equation \eqref{schwarziander} is a straightforward computation, see also \cite{brioschi}. Equation \eqref{dhnaa} follows from \cite{takhtajan}. The remaining identities are also readily verified.
\end{proof}
\section{Group determinants and autonomous systems}
\label{seckasner}
In this section we exhibit a link between some differential system first considered in \cite[eq.~6]{kasner}, see also \cite{kinyon}, in the context of finding solutions to the Einstein's equations, and the group-determinant of the cyclic group $\Z/3\Z$. Then we show that such an equation can be naturally put into a more general setting, and solve the resulting problem.

Let us fix $G$ a finite group of order $\sharp(G)=n\geqslant3$ and $x_g$ a set of independent variables indexed by the elements of the group $G$, and let us  define the group determinant or Frobenius determinant as
$$F(G)((x_g)_{g\in G})=\det(x_{gh^{-1}}).$$ 
For $G$ an Abelian group we have
\begin{theorem}[Dedekind, Frobenius, \cite{conrad1998}]
\label{dedekind}
Given a finite abelian group $G$ one has
$$F(G)((x_g)_{g\in G}):=\det(x_{gh^{-1}})=\displaystyle\prod_{\chi\in\widehat{G}}(\sum_{g\in G}\chi(g)x_g)$$ 
where $\widehat{G}$ is the character group of $G$, i.e. the group of homomorphisms from $G\to \C^\times$.
\end{theorem}
\begin{example}
\label{ex}
Let $G=\Z/n\Z$ then the group determinant reduces to the circulant determinant defined as
$$F(\Z/n\Z)(x_1,\ldots,x_n)=\displaystyle \left|\begin{array}{ccccc}x_1 & x_2 & x_3 & \ldots & x_{n} \\x_{n} & x_1& x_2 & \ldots & x_{n-1} \\\vdots & \vdots & \vdots & \vdots & \vdots \\x_2 & x_3 & x_4 & \ldots & x_1\end{array}\right|=\prod_{k=1}^n\left(\sum_{j=1}^n\zeta^{(k-1)(j-1)}x_j\right),$$
where $\zeta$ is a primitive $n$-th root of unity. In particular for $G=\Z/3\Z$ we obtain
$$F(\Z/3\Z)(x_1,x_2,x_3)=x_1^3+x_2^3+x_3^3-3x_1x_2x_3.$$
For the Klein group $G=\Z/2\Z\times\Z/2\Z=\{(0,0),(1,0),(0,1),(1,1)\}$ one readily sees that its character table is given by 
$$\begin{tabular}{|c|c|c|c|c|}\hline $\Z/2\Z\times\Z/2\Z$ & (0,0) & (1,0) & (0,1) & (1,1) \\\hline $\chi_{(0,0)}$ & 1 & 1 & 1 & 1 \\\hline $\chi_{(1,0)}$ & 1 & -1 & 1 & -1 \\\hline $\chi_{(0,1)}$ & 1& 1 & -1 & -1 \\\hline $\chi_{(1,1)}$ & 1 & -1 & -1 & 1 \\\hline \end{tabular}.$$
Thus we get
\begin{equation*}
\begin{split}
F(\Z/2\Z\times\Z/2\Z)(x_1,x_2,x_3,x_4)&=(x_1+x_2+x_3+x_4)(x_1-x_2+x_3-x_4)\\&(x_1+x_2-x_3-x_4)
(x_1-x_2-x_3+x_4)\end{split}
\end{equation*}
after identifying $(0,0)$ with $1$, $(1,0)$ with $2$, $(0,1)$ with $3$ and $(1,1)$ with $4$.
\end{example}
We now introduce the sought after Kasner's system \cite[eq.~6]{kasner}, see also \cite[\S3]{markus}. It is given by  
\begin{equation}
\label{kasner1}
\begin{split}
&x_1^\prime=x_2x_3-x_1^2\\
&x_2^\prime=x_3x_1-x_2^2\\
&x_3^\prime=x_1x_2-x_3^2,\quad^\prime=\dfrac{d}{dt}.
\end{split}
\end{equation}
This system describes solutions to Einstein's gravitation equations in a very special case. Its link with the group-determinant of $\Z/3\Z$ is given by the following
\begin{equation}
\label{kasner2}
\begin{split}
&x_1^\prime=-\dfrac{1}{3}\dfrac{\partial F(\Z/3\Z)}{\partial x_1}\\
&x_2^\prime=-\dfrac{1}{3}\dfrac{\partial F(\Z/3\Z)}{\partial x_2}\\
&x_3^\prime=-\dfrac{1}{3}\dfrac{\partial F(\Z/3\Z)}{\partial x_3}.
\end{split}
\end{equation}
Below we wish to study the system analogous to \eqref{kasner2} given by
\begin{equation}
\label{kasner3}
\dfrac{dx_k}{dt}=-\dfrac{1}{n}\dfrac{\partial F(\Z/n\Z)}{\partial x_k},\quad n\geqslant3,\;k=1,2,\ldots,n,
\end{equation}
together with a similar system derived from $F(\Z/2\Z\times\Z/2\Z)$. Before continuing we introduce, following \cite{kinyon}, the
\begin{definition}
\label{defkinyon}
Let $x^\prime=F(x),\;^\prime=\dfrac{d}{dt}$ be an autonomous system on an open, connected neighborhood of $0$ in $\C^n$ $n\geqslant3$, with $F$ analytic on $U$, and $F\not=0$. If another differential system $x^\prime=G(x)$ (with $G$ analytic) is given on $V\subset \C^m$, $m\geqslant1$, and there is a nonempty open subset $\widetilde{U}\subset U$, together with an analytic map $\varphi:\widetilde{U}\to V$ that sends parametrized solutions of $x^\prime=F(x)$ to parametrized solutions of $x^\prime=G(x)$, then we call $\varphi$ a solution preserving map from $x^\prime=F(x)$ to $x^\prime=G(x)$. If $V=U$, $G=F$, and $\varphi$ is (locally) invertible then $\varphi$ is called a symmetry of $x^\prime=F(x)$. If $V=U$, $G=\mu F$ (with $0\not=\mu:U\to\C$ analytic), and $\varphi$ is (locally) invertible then $\varphi$ is called an orbital symmetry of $x^\prime=F(x)$.
\end{definition}
With this definition in hand we have the
\begin{theorem}
\label{kasnerth1}
The autonomous system \eqref{kasner3} can be transformed into the system
\begin{equation}
\label{kasner4}
\begin{split}
&\dfrac{dy_1}{dt}=-y_2y_3\ldots y_n\\
&\dfrac{dy_k}{dt}=-\prod_{j=1,\;j\not= n+2-k}^ny_j,\quad k=2,3,\ldots,n,
\end{split}
\end{equation}
via the solution preserving map
\begin{equation}
\label{kasner5}
y_k:=\sum_{j=1}^n\zeta^{(k-1)(j-1)}x_j,\quad k=1,2,\ldots,n,
\end{equation}
with inverse given by
\begin{equation}
\label{kasner6}
x_k:=\dfrac{1}{n}\sum_{j=1}^n\zeta^{-(k-1)(j-1)}y_j,\quad k=1,2,\ldots,n.
\end{equation}
The system \eqref{kasner4} admits the quadratic first integrals (constant along the solutions) 
\begin{equation}
\label{kasner16000}
y_1^2-y_2y_n,y_1^2-y_3y_{n-1},\ldots,y_1^2-y_{\lfloor \frac{n+2}{2} \rfloor}y_{n+1-\lfloor \frac{n}{2} \rfloor}.
\end{equation}
Furthermore $y_2$ and $y_n$ are always Darboux polynomials for \eqref{kasner4}. Finally one can solve the system \eqref{kasner4} explicitly.
\end{theorem}
\begin{proof}
Using \eqref{kasner5} we obtain
\begin{equation}
\label{kasner7}
\begin{split}
\dfrac{dy_k}{dt}&=\sum_{j=1}^n\zeta^{(k-1)(j-1)}\dfrac{d x_j}{dt}=-\dfrac{1}{n}\sum_{j=1}^n\zeta^{(k-1)(j-1)}\dfrac{\partial F(\Z/n\Z)}{\partial x_j}\\
&=-\dfrac{1}{n}\sum_{j=1}^n\zeta^{(k-1)(j-1)}\sum_{l=1}^n\dfrac{\partial F(\Z/n\Z)}{\partial y_l}\dfrac{\partial y_l}{\partial x_j}\\
&=-\dfrac{1}{n}\sum_{j,l=1}^n\zeta^{(k-1)(j-1)}\dfrac{\partial F(\Z/n\Z)}{\partial x_l}\zeta^{(l-1)(j-1)}.
\end{split}
\end{equation}
Thus
\begin{equation}
\label{kasner8}
\dfrac{dy_k}{dt}=-\dfrac{1}{n}\sum_{l=1}^n\sum_{j=1}^n\zeta^{(k+l-2)(j-1)}\dfrac{\partial F(\Z/n\Z)}{\partial x_l}.
\end{equation}
Now since $\zeta$ is a primitive $n$-th root of unity we have
\begin{equation*}
\begin{split}
\sum_{j=1}^n\zeta^{a(j-1)}=0
\end{split}
\end{equation*}
if $a$ is not divisible by $n$ and 
\begin{equation*}
\begin{split}
\sum_{j=1}^n\zeta^{a(j-1)}=n
\end{split}
\end{equation*}
otherwise. From this there results
\begin{equation}
\label{kasner9}
\begin{split}
&\dfrac{dy_1}{dt}=-\dfrac{\partial F(\Z/n\Z)}{\partial y_1}\\
&\dfrac{dy_k}{dt}=-\dfrac{\partial F(\Z/n\Z)}{\partial y_{n+2-k}},\quad k\geqslant2.
\end{split}
\end{equation}
From \exref{ex} it follows that $F(\Z/n\Z)(x_1,x_2,\ldots,x_n)=y_1y_2\ldots y_n$. Hence by \eqref{kasner9} we find the system \eqref{kasner4}.

 Let us now solve the system \eqref{kasner4}. There are three cases. The first case happens when two or more of the $y_k$, $1\leqslant k\leqslant n$ are zero. Then one readily sees from the system \eqref{kasner4} that all its solutions are constant. In the second case exactly one of the $y_k$, $1\leqslant k\leqslant n$ is zero. Then we observe from the system \eqref{kasner4} that $y_1$ is not zero. Thus we must have $y_m=0$ for some fixed $m>1$. From \eqref{kasner4}, every $y_k$ is constant, except $y_{n+2-m}$. Also since only $y_m=0$, it is clear that $n+2\not=2m$. Thus we have with fixed $m$, $m\not=1$, and $m\not=\dfrac{1}{2}(n+2)$
 \begin{equation*}
 \begin{split}
 &y_k=c_k\qquad(k=1,2,\ldots,n,\quad k\not=n+2-m\not=m,\;c_m=0)\\
 &\dfrac{dy_{n+2-m}}{dt}=ay_{n+2-m},\quad a=\prod_{j\not=m,n+2-m}c_j\not=0.
 \end{split}
 \end{equation*}
 This readily gives
 $$y_{n+2-m}=c_{n+2-m}e^{at},\quad c_{n+2-m}\not=0.$$
 Let us now examine the remaining case where no $y_k$ vanishes identically. From \eqref{kasner4} we deduce
 \begin{equation}
 \label{kasner10}
 \begin{split}
 y_{n+2-k}dy_k=y_1dy_1\quad(k=2,\ldots,n)
 \end{split}
 \end{equation}
 and from this in turn
 $$y_{n+2-k}dy_k-y_kdy_{n+2-k}=0,\quad k=2,3,\ldots,m:=\Big\lfloor \dfrac{n+1}{2} \Big\rfloor.$$
 Integration yields
 \begin{equation}
 \label{kasner11}
 y_{n+2-k}=a_k^{-1}y_k,\quad(k=2,3,\ldots,m),\;a_k\in\C^\times.
 \end{equation}
 Inserting \eqref{kasner11} into \eqref{kasner10} we find
 $$y_kdy_k=a_ky_1dy_1,\quad(k=2,3,\ldots,m).$$
 This gives
 \begin{equation}
\label{kasner12}
y_k^2=a_k(y_1^2+b_k), \quad(k=2,3,\ldots,m).
\end{equation}
Here the $b_k$ are complex constants. In the case $n$ odd, \eqref{kasner11} and \eqref{kasner12} are sufficient for the satisfaction of \eqref{kasner10}. For $n$ even we necessarily have $n=2m$, with the above choice of $m$. In this case we also have from \eqref{kasner10} with $k=m+1$,
$$y_{m+1}dy_{m+1}=y_1dy_1$$
from which we deduce
\begin{equation}
\label{kasner13}
y_{m+1}^2=y_1^2+a, \quad a\in\C.
\end{equation}
The previous reasonings allow us to express $y_2$, $y_3$, $\ldots y_n$ in terms of $y_1$ alone. Therefore we may reduce the first of equations \eqref{kasner4} to an equation involving $y_1$ only. Indeed combining \eqref{kasner11} and \eqref{kasner12} yields
$$y_ky_{n+2-k}=y_1^2+b_k,\quad k=2,3,\ldots,m.$$
Forming the product of these expressions and using \eqref{kasner4} we find
\begin{equation}
\label{kasner14}
\begin{split}
\dfrac{dy_1}{dt}=-y_{m+1}\prod_{k=2}^m(y_1^2+b_k)\\
y_{m+1}^2=y_1^2+a,
\end{split}
\end{equation}
where the factor $y_{m+1}$ must be omitted for the case $n$ odd. Let us choose the numbers $b_2$, $b_3$, $\ldots b_m$ as distinct then we may write
\begin{equation}
\label{kasner15}
\dfrac{1}{\prod_{k=2}^m(y_1^2+b_k)}=\sum_{k=2}^m\dfrac{A_k}{y_1^2+b_k},\qquad A_k^{-1}=\prod_{j=2,\,j\not=k}^m(b_j-b_k).
\end{equation}
If $n$ is even and we use \eqref{kasner15} in \eqref{kasner14} there results
\begin{equation}
\label{kasner16}
\sum_{k=2}^m\dfrac{A_kdy_1}{(y_1^2+b_k)y_{m+1}}=-dt;
\end{equation}
while if $n$ is odd we obtain
\begin{equation}
\label{kasner17}
\sum_{k=2}^m\dfrac{A_kdy_1}{(y_1^2+b_k)}=-dt.
\end{equation}
The verification that the expressions \eqref{kasner16000} are constant along the solutions of \eqref{kasner4} is immediate. Finally one readily sees by considering the associated vector field to \eqref{kasner4}, that $y_2$ and $y_n$ are Darboux polynomials for \eqref{kasner4}.
\end{proof}
\begin{example}
\label{exkasner}
Let us look at an example of the \theoref{kasnerth1}. We treat the case $n=3$. Then \eqref{kasner3} is identical to \eqref{kasner1} given by
\begin{equation*}
\label{}
\begin{split}
&x_1^\prime=x_2x_3-x_1^2\\
&x_2^\prime=x_3x_1-x_2^2\\
&x_3^\prime=x_1x_2-x_3^2,\quad^\prime=\dfrac{d}{dt}.
\end{split}
\end{equation*}
 The associated system \eqref{kasner4} reads
 \begin{equation}
\label{kasner1500}
\begin{split}
&y_1^\prime=-y_2y_3\\
&y_2^\prime=-y_1y_2\\
&y_3^\prime=-y_1y_3,\quad^\prime=\dfrac{d}{dt}.
\end{split}
\end{equation}
Furthermore \eqref{kasner1500} possesses the quadratic first integral $y_1^2-y_2y_3$. We remark that \eqref{kasner1} or its associated vector field on $\C^3$: $(x_2x_3-x_1^2)\dfrac{\partial}{\partial x_1}+(x_1x_3-x_2^2)\dfrac{\partial}{\partial x_2}+(x_1x_2-x_3^3)\dfrac{\partial}{\partial x_3}$ admits the symmetric group $S_3$, as symmetry group (see \cite[ex.~2.3]{kinyon}), in the sense of \defiref{defkinyon}. Hence by \cite[\S2.4.1]{guillot}, $\xi(t):=x_1(t)+x_2(t)+x_3(t)$ should satisfy an equation of the Chazy type of the form
\begin{equation}
\label{kasner15000}
\xi^{\prime\prime\prime}=a\xi^4+b\xi^2\xi^\prime+c(\xi^\prime)^2+d\xi,\quad^\prime=\dfrac{d}{dt},\;a,b,c,d\in\C.
\end{equation}
But in the case of \eqref{kasner1} this is easily deduced. Indeed introducing the elementary symmetric functions on $x_1$, $x_2$, $x_3$ 
 \begin{equation}
\label{kasner1501}
\begin{split}
&\xi(x)=x_1+x_2+x_3\\
&\eta(x)=x_1x_2+x_1x_3+x_2x_3\\
&\zeta(x)=x_1x_2x_3,
\end{split}
\end{equation}
we have by \cite[eq.~7]{kasner} the following system
\begin{equation}
\label{kasner1502}
\begin{split}
&\xi^\prime=\xi^2+\eta\\
&\eta^\prime=0\\
&\zeta^\prime+3\xi\zeta=\eta^2,\quad^\prime=\dfrac{d}{dt}.
\end{split}
\end{equation}
If we derive the first equation of the system \eqref{kasner1502} twice we readily obtain an equation of the advertised form \eqref{kasner15000}.
\end{example}
\begin{theorem}
Let $G=\Z/2\Z\times\Z/2\Z$ be the Klein group and 
\begin{equation}
\label{kasner18}
\begin{split}
F(\Z/2\Z\times\Z/2\Z)(x_1,x_2,x_3,x_4)&=(x_1+x_2+x_3+x_4)(x_1-x_2+x_3-x_4)\\&(x_1+x_2-x_3-x_4)
(x_1-x_2-x_3+x_4).
\end{split}
\end{equation}
Then the differential system 
\begin{equation}
\label{kasner19}
\dfrac{d x_k}{dt}=\dfrac{\partial F(\Z/2\Z\times\Z/2\Z) }{\partial x_k},\quad 1\leqslant k\leqslant4
\end{equation}
admits the symmetric group $S_4$ (the group of linear transformations defined by permutations of the standard basis) as a group of symmetry in the sense of \defiref{defkinyon}. Furthermore the mapping
\begin{equation}
\label{kasner20}
x:=(x_1,x_2,x_3,x_4)\mapsto\left(\begin{array}{c}\varphi_1(x) \\\varphi_2(x) \\\varphi_3(x) \\\varphi_4(x)\end{array}\right),
\end{equation}
where
\begin{equation}
\label{kasner21}
\begin{split}
\varphi_1(x)=x_1+x_2+x_3+x_4\\
\varphi_2(x)=x_1x_2+x_1x_3+x_1x_4+x_2x_3+x_2x_4+x_3x_4\\
\varphi_3(x)=x_1x_2x_3+x_1x_2x_4+x_1x_3x_4+x_2x_3x_4\\
\varphi_4(x)=x_1x_2x_3x_4,
\end{split}
\end{equation}
are the elementary symmetric polynomials formed on the variables $x_1$, $x_2$, $x_3$, $x_4$, is a solution preserving from \eqref{kasner19} to another autonomous polynomial differential system $y^\prime=G(y),\,^\prime=\frac{d}{dt}$, $y\in\C^4$.
\end{theorem}
\begin{proof}
According to \cite[prop.~1.1]{kinyon}, $S_4$ (considered as embedded in $\rm{GL}(4,\C)$ via its permutation representation), is a symmetry group of \eqref{kasner19} if and only if for any $T\in S_4$ we have 
\begingroup
%\begingroup % or {
\setlength{\abovedisplayskip}{0pt}
\setlength{\belowdisplayskip}{2pt}
\setlength{\belowdisplayshortskip}{-2pt}
\begin{equation}
\label{kasner22}
T\circ\left(\begin{array}{c}\dfrac{\partial F(\Z/2\Z\times\Z/2\Z) }{\partial x_1} \\\dfrac{\partial F(\Z/2\Z\times\Z/2\Z) }{\partial x_2} \\\dfrac{\partial F(\Z/2\Z\times\Z/2\Z) }{\partial x_3} \\\dfrac{\partial F(\Z/2\Z\times\Z/2\Z) }{\partial x_4}\end{array}\right)\circ T^{-1}=\left(\begin{array}{c}\dfrac{\partial F(\Z/2\Z\times\Z/2\Z) }{\partial x_1} \\\dfrac{\partial F(\Z/2\Z\times\Z/2\Z) }{\partial x_2} \\\dfrac{\partial F(\Z/2\Z\times\Z/2\Z) }{\partial x_3} \\\dfrac{\partial F(\Z/2\Z\times\Z/2\Z) }{\partial x_4}\end{array}\right).
\end{equation}
\endgroup
Since the symmetric group $S_4$ is generated by transpositions it suffices to verify such an identity for any transposition in order to indeed show that $S_4$ is a symmetry group of \eqref{kasner19}. Let $\tau_{ij},\,i<j$ denote the permutation exchanging $i$ and $j$. One verifies by a long but straightforward computation that
\begingroup % or {
\setlength{\abovedisplayskip}{0pt}
\setlength{\belowdisplayskip}{2pt}
\setlength{\belowdisplayshortskip}{-2pt}
\begin{equation}
\label{kasner23}
\begin{split}
\left(\begin{array}{c}\dfrac{\partial F(\Z/2\Z\times\Z/2\Z) }{\partial x_1} \\\dfrac{\partial F(\Z/2\Z\times\Z/2\Z) }{\partial x_2} \\\dfrac{\partial F(\Z/2\Z\times\Z/2\Z) }{\partial x_3} \\\dfrac{\partial F(\Z/2\Z\times\Z/2\Z) }{\partial x_4}\end{array}\right)\circ\tau_{ij}^{-1}=\left(\begin{array}{c}\dfrac{\partial F(\Z/2\Z\times\Z/2\Z) }{\partial x_{\tau_{ij}(1)}} \\\dfrac{\partial F(\Z/2\Z\times\Z/2\Z) }{\partial x_{\tau_{ij}(2)}} \\\dfrac{\partial F(\Z/2\Z\times\Z/2\Z) }{\partial x_{\tau_{ij}(3)}} \\\dfrac{\partial F(\Z/2\Z\times\Z/2\Z) }{\partial x_{\tau_{ij}(4)}}\end{array}\right)
\end{split}
\end{equation}
\endgroup
Post-composing the identities in \eqref{kasner23} with the corresponding transpositions, one readily sees that for any transposition $T$ of $S_4$ one has 
\begin{equation}
\label{kasner24}
T\circ \left(\begin{array}{c}\dfrac{\partial F(\Z/2\Z\times\Z/2\Z) }{\partial x_1} \\\dfrac{\partial F(\Z/2\Z\times\Z/2\Z) }{\partial x_2} \\\dfrac{\partial F(\Z/2\Z\times\Z/2\Z) }{\partial x_3} \\\dfrac{\partial F(\Z/2\Z\times\Z/2\Z) }{\partial x_4}\end{array}\right)\circ T^{-1}=\left(\begin{array}{c}\dfrac{\partial F(\Z/2\Z\times\Z/2\Z) }{\partial x_1} \\\dfrac{\partial F(\Z/2\Z\times\Z/2\Z) }{\partial x_2} \\\dfrac{\partial F(\Z/2\Z\times\Z/2\Z) }{\partial x_3} \\\dfrac{\partial F(\Z/2\Z\times\Z/2\Z) }{\partial x_4}\end{array}\right).
\end{equation} 
Hence one may infer that \eqref{kasner22} is true for any permutation of $S_4$ and that $S_4$ is a symmetry group of \eqref{kasner19}. 

In order to obtain the last part of the theorem we may use \cite[th.~2.1]{kinyon}. Indeed $S_4\subset \rm{GL}(4,\C)$ has a set of polynomial invariants finitely generated by the elementary symmetric polynomials $\varphi_1(x)$, $\varphi_2(x)$, $\varphi_3(x)$, $\varphi_4(x)$; furthermore as we just saw in the first part of this proof 
$$T\circ \left(\begin{array}{c}\dfrac{\partial F(\Z/2\Z\times\Z/2\Z) }{\partial x_1} \\\dfrac{\partial F(\Z/2\Z\times\Z/2\Z) }{\partial x_2} \\\dfrac{\partial F(\Z/2\Z\times\Z/2\Z) }{\partial x_3} \\\dfrac{\partial F(\Z/2\Z\times\Z/2\Z) }{\partial x_4}\end{array}\right)\circ T^{-1}=\left(\begin{array}{c}\dfrac{\partial F(\Z/2\Z\times\Z/2\Z) }{\partial x_1} \\\dfrac{\partial F(\Z/2\Z\times\Z/2\Z) }{\partial x_2} \\\dfrac{\partial F(\Z/2\Z\times\Z/2\Z) }{\partial x_3} \\\dfrac{\partial F(\Z/2\Z\times\Z/2\Z) }{\partial x_4}\end{array}\right),\quad \forall \,T\in S_4.$$
Then according to \cite[th.~2.1]{kinyon} the map
$$x:=(x_1,x_2,x_3,x_4)\mapsto\left(\begin{array}{c}\varphi_1(x) \\\varphi_2(x) \\\varphi_3(x) \\\varphi_4(x)\end{array}\right)$$
is solution preserving from \eqref{kasner19} to a polynomial differential system $y^\prime=G(y)$.
\end{proof}
\section{Wirtinger connections}
\label{wirtinger}
%\section{Projective and affine connections, Bergman Kernel and Wirtinger connection: link with theta functions}
In this section we study various connections. We start by recalling some classical definitions on connections \cite{GusP},\,\cite{Tyurin}. 
\begin{enumerate}
\item For a non-vanishing function $f= f(t)$, we write
\begin{eqnarray*} 
\{f, t\}_{1}&:=&\frac{d}{dt} \log \frac{d f}{d t},\\
\{f, t\}_2 &:= &\frac{d^2}{dt^2} \log \frac{d f}{d t}- \frac{1}{2}
( \frac{d}{dt} \log \frac{d f}{d t})^2,
\end{eqnarray*}
$\{f, t\}_2$ is the Schwarzian derivative of  $f$.
\item If  $X$ is a Riemann surface with a system of coordinate charts
  $\{(U,z)\}$, a system of 
functions $\{(q, U, z)\}$ (one function for each coordinate chart) defines a $k$-form or a tensor of order $k$ if, under change of coordinates it transforms according to
$$ {\tilde q}(\tilde z)(\frac{d {\tilde z}}{d z})^k=  q(z).$$
It defines a 
projective connection on  $X$ if it transforms as:
$${\tilde q}(\tilde z)(\frac{d {\tilde z}}{d z})^2 = q(z) - 
\{ {\tilde z}, z\}_2 $$
and defines an affine connection if 
$${\tilde q}(\tilde z)\frac{d {\tilde z}}{d z}= q(z)- 
\{{\tilde z}, z\}_1.$$
%\item If $f$ is a $1$-form or more generally a form of bidegree $(1,*),\, 
% \frac{1}{2} \frac{\partial}{\partial z}\log f$ is an affine connection
 %on $X \setminus f^{-1}(0)$. In particular if $\displaystyle ds= \frac{1}{\omega(z)}|d z|$ is a hermitian metric sur $X$,
%$-\frac{\partial}{\partial z}\log \omega $ is a well defined affine connection.
\item If $r$ is an affine connection, the curvature $$q= \frac{\partial r}{\partial
  z}- \frac{1}{2}r^2$$ is a projective connection. We thus have  a way to obtain projective connections from 
affine connections.
 \item An affine connection $r$ defines a covariant derivative
$ \nabla = \nabla_{\alpha}$ mapping $\alpha$-forms into
$\alpha +1$-forms by
$$\nabla f =  \frac{\partial f}{\partial z} - \alpha r f. $$
\end{enumerate}
To give an example, let us consider the Eisenstein series $E_2$, which verifies (see \cite{zag2004, serre1970}) the functional equation
$$\displaystyle E_2(\frac{a\tau+ b}{c\tau+ d})(c\tau+ d)^{-2}= E_2(\tau)+ \frac{6}{i\pi} \frac{c}{c\tau+ d}.$$ 
This shows that $ \displaystyle \frac{\pi i}{3} E_2$ defines an affine connection on $\mathfrak H/SL(2,\Z)$ minus the $SL(2,\Z)$-orbits of the points $i\infty$, $i$ and $e^{2i\pi/3}$, see \cite[appen.~C]{Dubrovin}. From the third relation in \eqref{S1} we find that the curvature is $\frac{1}{18} \pi^2 E_4$, which is a projective connection on $\mathfrak H/SL(2,\Z)$ minus the $SL(2,\Z)$-orbits of the points $i\infty$, $i$ and $e^{2i\pi/3}$. The covariant differentiation defined by this affine connection corresponds to the fact that the Serre derivative $$\frac{1}{2i\pi}\frac{d}{d\tau} - \frac{2k}{12} E_2=\frac{1}{2i\pi}\frac{d}{d\tau} - \frac{k}{6} E_2$$ sends the space of modular forms of weight $2k$ for ${SL(2,\Z)}$ into the space of modular forms of weight $2k+ 2$.

Suppose now that $X$ is a compact Riemann surface of genus $g >0$. We fix a homology canonical basis $a_1,\cdots, a_g,~ b_1,\cdots, b_g$ and a basis $w_1,\ldots,w_g$ of holomorphic differential such that for $i,j= 1, \cdots, g$
$$\int_{a_j} w_j= \delta_i^j,\quad \int_{b_j}= \tau_{ij},$$
where $(\tau_{ij})_{1\leq i,j\leq g}$ is a Riemann matrix and $\delta_i^j$ is the Kronecker symbol.

In this situation, there is a unique symmetric bidifferential of the second kind $\omega$, called the Bergman Kernel of $X$, such that
\begin{enumerate}
 \item $\omega= \frac{1}{(u- v)^2 } du dv+ H(u, v) du dv$, where for $U\subset X$ an open set, $H(u,v)$ is holomorphic on $U\times U\subset X\times X$. Here $u$ designates a coordinate on the first direct factor $U$, and $v$ is a coordinate on the second direct factor U.
 \item $\rm{Bires}(\omega)=1$, zero residue, and no further singularities.
\item $\int_{a_i} \omega(u,v) =0$ for any fixed $u\in X$ and $i= 1,\cdots,g$,
\item $\int_{b_i} \omega(u,v)= w_i(u),\, i= 1,\cdots,g$.
\end{enumerate}
\noindent It is a general fact \cite[prop.~1.3.6]{Tyurin}, \cite[p.~20]{Fay}, that a symmetric bidifferential of the second kind $\omega$, given along the diagonal by $\displaystyle \omega=  \frac{1}{(u- v)^2} du dv+ H(u,v)du dv$, with a holomorphic $H(u, v)$ gives rise to a projective connection $h_{\omega}(Z)= 6 H(u,v)_{\vert u=v= Z}$.

 For $\epsilon, \,\delta \in \BR^g$ we define the genus $g\geqslant1$ theta function with characteristics by
$$\theta \begin{bmatrix}
  \epsilon\\
\delta
 \end{bmatrix}(z,\tau)= \sum_{n\in \Z^g} e^{i\pi(n+ \frac{\epsilon}{2})\tau(n+ \frac{\epsilon}{2})^t+ 2i\pi(n+ \frac{\epsilon}{2})(z+ \frac{\delta}{2})^t},
$$
where $z\in\C^g$ and $\tau$ is a symmetric $g\times g$ complex matrix, with positive definite imaginary part $\rm{Im}(\tau_{ij})$, a so-called Riemann matrix. We have 
$$\theta \begin{bmatrix}
  \epsilon\\
\delta
 \end{bmatrix}(z,\tau)=  
 (-1)^{\epsilon\delta^t} \theta \begin{bmatrix}
  \epsilon\\
\delta
 \end{bmatrix}(-z,\tau).$$
%For instance, in terms of characteristics, the four Jacobi theta functions are given as follows
%\begin{equation}
%\label{ser43}
%\begin{split}
%\vartheta_1 (z, \tau)&= 2\sum_0 ^{\infty} q^{(n+ \frac{1}{2})^2} \cos(2n+1)\pi z= \theta\begin{bmatrix}
%  1\\
%0
 %\end{bmatrix}(z,\tau)\\
%\vartheta_2 (z,\tau)&= 1+2 \sum_0 ^{\infty} (-1)^n q^{n^2} \cos 2n\pi z=  \theta\begin{bmatrix}
%  0\\
%1
% \end{bmatrix}(z,\tau) \\
%\vartheta_3 (z,\tau)&= 1+2 \sum_0 ^{\infty} q^{n^2} \cos2 n\pi z=  \theta\begin{bmatrix}
  %0\\
%0
% \end{bmatrix}(z,\tau) \\
%\vartheta (z,\tau)&= 2 \sum_0 ^{\infty} (-1)^n q^{(n+ \frac{1}{2})^2} \sin (2n+1)\pi z=  \theta\begin{bmatrix}
%  1\\
%1
 %\end{bmatrix}(z,\tau),\quad q=e^{i\pi\tau}.
% \end{split}
%\end{equation}
Every even $\theta$-function with characteristic 
$\displaystyle \begin{bmatrix}
  \epsilon\\
\delta
 \end{bmatrix}
$ on $X$ (for example the first $3$ cases in \eqref{ser43}) with non-zero $\theta$-constant $\theta \begin{bmatrix}
  \epsilon\\
\delta
 \end{bmatrix}
(0, \tau) \neq 0$, determines a bidifferential $\omega_{\begin{bmatrix}
  \epsilon\\
\delta
 \end{bmatrix}}
$, in the following way \cite[p.~20, 22]{Fay} or \cite{Tyurin}. We first observe that
$$\sum_{i,j= 1}^g w_i(u) w_j(v) \frac{\partial^2}{\partial z_i \partial z_j} \log\theta \begin{bmatrix}
  \epsilon\\
\delta
 \end{bmatrix}
(z, \tau)_{|z= 0}$$
is a holomorphic symmetric bidifferential on $X\times X$; and that
$$ \omega+ \sum_{i,j= 1}^g  \frac{\partial^2}{\partial z_i \partial z_j} \log\theta \begin{bmatrix}
  \epsilon\\
\delta
 \end{bmatrix}
(z, \tau)_{|z= 0} w_i(u) w_j(v)=:\omega_{\begin{bmatrix}
  \epsilon\\
\delta
 \end{bmatrix}}
$$
is invariant under a change of homology basis preserving the characteristic $\begin{bmatrix}
  \epsilon\\
\delta
 \end{bmatrix}
$
of the $\theta$-function. We have
\begin{definition}
The symmetric invariant bidifferential of the second kind $\omega_{\begin{bmatrix}
  \epsilon\\
\delta
 \end{bmatrix}}$ is called a Klein bidifferential. The associated projective connection is called the Wirtinger connection $W_{\begin{bmatrix}
  \epsilon\\
\delta
 \end{bmatrix}}$. Finally writing for short $\beta= \begin{bmatrix}
  \epsilon\\
\delta
 \end{bmatrix}$, the invariant Klein bidifferential $K_X(u, v)$ on the Riemann surface $X$ is defined by
$$K_X(u,v)= \omega(u, v)+ \frac{2}{4^g+ 2^g} \sum_{i,j= 1}^g \frac{\partial^2}{\partial z_i \partial z_j} \log \prod_{\beta,\, even} 
\theta [\beta](z,\tau)_{|z= 0} w_i(u) w_j(v).
$$
The associated projective connection $W_X$ is called the invariant Wirtinger connection. Hence $K_X$ is an averaged bidifferential over all the $2^{g-1}( 2^g +1)$ even characteristics. 
 \end{definition}
Let us now give the Wirtinger connections in the genus one case. We recall that classically
\begin{eqnarray}\label{Wei}
\zeta '(u)= -\wp(u),\quad \eta_1= \zeta (\omega_1),\quad
\eta_2= \zeta(\omega_2),\quad\tau= \frac{\omega_2}{\omega_1 },\quad\rm{Im}\, \tau>0.
 \end{eqnarray}
 As is well known, if we identify the opposite sides of the parallelogram generated by 
$2\omega_1, 2\omega_2, \frac{\omega_2}{\omega_1}\in\mathfrak H$ in the complex $u$-plane, we obtain a genus one Riemann surface, denoted hereafter $X$. The effect of a closed $a$-cycle is the parallel displacement by the vector $2\omega_1$ and the effect of a closed $b$-cycle is the displacement by the vector $2\omega_2$. With the notations of \eqref{Wei}, the normalized integral of the first kind and the Riemann matrix are given respectively by
$$ w= \frac{1}{2\omega_1}u, \,u= u(P); \,\,\;\; \tau= \int_0^{2\omega_2} dw.$$
Further \cite[p.~428]{jordan1896}
\begin{equation}
\begin{split}
&\zeta(u)=\dfrac{\eta_1}{\omega_1}u+\dfrac{1}{2\omega_1}\dfrac{d}{dw}\{\log\vartheta(w,\tau)\}\\
&\wp(u)=\left(\dfrac{1}{2\omega_1}\right)^2\left[-4\eta_1\omega_1-\dfrac{d^2}{dw^2}(\log\vartheta(w,\tau))\right]\\
&\wp(u+\omega_i)=\left(\dfrac{1}{2\omega_1}\right)^2\left[-4\omega_1\eta_1-\dfrac{d^2}{dw^2}(\log\vartheta_i(w,\tau))\right],\;i=1,2,3.
\end{split}
\end{equation}
The function $$t(P,Q)=t(u,v)= \zeta( u-v)- \frac{\eta_1}{\omega_1}u,\,\,\;u=u(P),\;v= v(Q),\quad (P,Q)\in X\times X$$
is a normalized integral of the second kind, $u$ is a coordinate near $P$ and $v$ is a coordinate near $Q$. It has a simple pole along the diagonal $\{P=Q\}$ and zero period over the $a$-cycle. The Bergman kernel in the genus one case is
\begin{equation}
\omega(P, Q)= \omega(u,v)=-\frac{d}{d u} t(u, v)dudv= \left(\wp(u- v; \omega_1, \omega_2) + \frac{\eta_1}{\omega_1}\right)dudv,\quad (P,Q)\in X\times X.
\end{equation}
It is a bidifferential on $X\times X$ and has a double pole along the diagonal $\{P= Q\}$ with zero residue and biresidue $1$ and the other properties we recalled before.

According for example to \cite[p. 415, p. 455]{jordan1896}, we have 
\begin{equation}
\label{jordanholder}
\displaystyle 2\eta_1 \omega_1= -\frac{1}{6} \frac{\vartheta '''(0) }{\vartheta '(0)}= \frac{\pi^2}{6} E_2.
\end{equation}
Using also $\wp(u-v)= \frac{1}{(u-v)^2}+ O(u-v)^2$ and the previous identity \eqref{jordanholder}, we find that the connection associated to the Bergman kernel is 
 $$S_B(P)=- \frac{1}{12\omega^2_1} \vartheta '''(0) / \vartheta'(0) =  \frac{\pi^2}{12\omega_1^2} E_2.$$ For $2\omega_1= 1, 2\omega_2= \tau$, 
the Bergman kernel is
 $\displaystyle \omega(u, v)= (\wp(u-v;1,\tau)+  \frac{\pi^2}{3} E_2)dudv$. 
 %which shows that the K-dV equation for the Weierstrass function $\wp''= 6\wp^2- \frac{1}{2} g_2$
 %and the third relation in \eqref{S1} take the following form
% $$-2i\pi \partial _{\tau} S_0= 2 K(u,v) S_0+ 2K(u,v) S_0+ \partial_{uu}^2 K(u,v)- 6 K(u,v)^2.$$
\begin{theorem}
\label{projconn}
The Wirtinger connections associated to the even characteristics $ \begin{bmatrix}
  0\\
0
 \end{bmatrix}$, $\begin{bmatrix}
  1\\
0
 \end{bmatrix}$, $\begin{bmatrix}
  0\\
1
 \end{bmatrix}
$ are respectively $-6e_1, -6e_2, -6e_3$.
\end{theorem}
\begin{proof}
We use the classical formulas for the Weierstrass function  $\wp(u)= \wp(u, \omega_1, \omega_2)$, with $u= 2\omega_1 w$
$$\wp(u+ \omega_i)= \left(\frac{1}{2\omega_1}\right)^2 \{ -4 \eta_1 \omega_1- \frac{ d^2\log \vartheta_i(w)}{dw^2}\}$$
and the local expansions 
$$\wp(u)= \frac{1}{(2\omega_1 w)^2}+ \frac{g_2}{20}(2\omega_1 w)^2+ \frac{g_3}{28}(2\omega_1 w)^4+\cdots$$
$$\wp(u+\omega_i)= \wp(\omega_i)+ \wp''(\omega_i)\frac{(2\omega_1 w)^2}{2}+ \cdots= e_i +(6e_i^2- \frac{1}{2} g_2)\frac{(2\omega_1 w)^2}{2}+\cdots$$
which give 
\begin{equation}
\begin{split}
 \wp(u)- \wp(u+ \omega_i)&= \wp(u)- \left(\frac{1}{2\omega_1}\right)^2 \{ -4 \eta_1 \omega_1- \frac{ d^2\log \vartheta_i(w)}{dw^2}\}\\
& = \wp(u)+\frac{\eta_1}{\omega_1}+\dfrac{1}{4\omega_1^2}\dfrac{d^2\log\vartheta_i(w)}{dw^2}=\frac{1}{u^2}- e_i+ O(u^2).
 \end{split}
 \end{equation}
\end{proof}
\begin{remark}
 Let $^\prime=\frac{d}{dw}$. The invariant Klein bidifferential, \cite[p. 35]{Fay}, is by definition 
 $$K_X=\omega(u,v)+ \frac{1}{3} \left(\log \vartheta_1 \vartheta_2 \vartheta_3\right)''_{|z=w=0}du dv= \omega(u,v)+ \frac{1}{3} \frac{\vartheta'''(0)}{\vartheta'(0)} du dv= \wp(u-v) du dv,$$
 by virtue of \cite[p.~416]{jordan1896} and the fact that 
 $$\vartheta_1^\prime(0,\tau)=\vartheta_2^\prime(0,\tau)=\vartheta_3^\prime(0,\tau)=0.$$ The invariant Wirtinger connection, associated with this invariant Klein bidifferential, which is the projective connection associated to this bidifferential, is equal to zero; because $\wp(u-v)= \frac{1}{(u-v)^2}+ O(u-v)^2$.
\end{remark}
%\section{Question sur l'existence d'une structure affine}
%Une variété complexe de dimension $n$ est dite admettre un atlas affine s'il admet un recouvrement ouvert $\{(U_i,\phi_j)\}_{j\in J}$ de telle sorte que $\phi_{jk}=\phi_j\circ\phi_k^{-1}$ soit une transformation affine complexe de $\C^n$ chaque fois que $U_j\cap U_k\not=\emptyset$. 

%Ma question est la suivante: on se donne une fonction homogène $F$ en $n$ variables définie sur un ouvert $U$ de $\C^n$, $n\geqslant1$ et de degré $d\geqslant1$, et on désigne par $G_F$ le sous groupe de $SL(n,\C)$ laissant $F$ invariant, i.e., $g\in G_F\iff F(g(x))=F(x)$, $\for x\in U$. Soit $M$ une composante connexe de l'hypersurface de niveau $F=1$. Supposons de plus que $G_F$ agit transitivement sur $M$, alors est-ce que $M$ admet un atlas affine?

%\begin{remark}
%It follows from the works of Sasaki that my question has a positive answer in the real category.
%\end{remark}
%\section{Frobenius determinants and the construction of certain Hurwitz spaces}

\end{document}